\documentclass[11pt]{amsart}

\usepackage{amscd}
\usepackage{amsmath, amssymb}
\usepackage{amsfonts}
\usepackage{enumerate}
\newcommand{\de}{\partial}
\newcommand{\db}{\overline{\partial}}
\newcommand{\dbar}{\overline{\partial}}

\newcommand{\ddbar}{\sqrt{-1} \partial \overline{\partial}}

\newcommand{\ov}[1]{\overline{#1}}

\newcommand{\ti}[1]{\tilde{#1}}
\newcommand{\vp}{\varphi}
\newcommand{\vol}{\mathrm{Vol}}

\newcommand{\ve}{\varepsilon}

\newcommand{\ol}{\overline}

\renewcommand{\leq}{\leqslant}
\renewcommand{\geq}{\geqslant}
\renewcommand{\le}{\leqslant}
\renewcommand{\ge}{\geqslant}

\newcommand{\be}{\begin{equation}}
\newcommand{\ee}{\end{equation}}

\begin{document}
\newcounter{remark}
\newcounter{theor}
\setcounter{remark}{0}
\setcounter{theor}{1}
\newtheorem{claim}{Claim}
\newtheorem{theorem}{Theorem}[section]
\newtheorem{lemma}[theorem]{Lemma}
\newtheorem{corollary}[theorem]{Corollary}
\newtheorem{proposition}[theorem]{Proposition}
\newtheorem{question}{question}[section]
\newtheorem{defn}{Definition}[theor]
\newtheorem{remark}[theorem]{Remark}

\numberwithin{equation}{section}


\title{The Fu-Yau equation on compact astheno-K\"{a}hler manifolds}

\author[J. Chu]{Jianchun Chu}
\address{Institute of Mathematics, Academy of Mathematics and Systems Science, Chinese Academy of Sciences, Beijing 100190, P. R. China}
\email{chujianchun@gmail.com}
\author[L. Huang]{Liding Huang}
\address{School of Mathematical Sciences, University of Science and Technology of China, Hefei 230026, P. R. China}
\email{huangld@mail.ustc.edu.cn}
\author[X. Zhu]{Xiaohua Zhu}
\address{School of Mathematical Sciences, Peking University, Yiheyuan Road 5, Beijing 100871, P. R. China }
\email{xhzhu@math.pku.edu.cn}



\begin{abstract}
In this paper, we study the Fu-Yau equation on compact Hermitian manifolds and prove the existence of solutions of equation on  astheno-K\"{a}hler manifolds. We also  prove   the uniqueness   of solutions  of Fu-Yau equation  when the slope parameter $\alpha$ is negative.
\end{abstract}
\maketitle

\section{Introduction}
Let $(M,\omega)$ be an $n$-dimensional compact  K\"{a}hler  manifold.    As a  reduced   generalized Strominger system in higher dimensions,  Fu and Yau introduced the following fully nonlinear equation for $\varphi$ \cite{FuY08},
\begin{equation}\label{Fu-Yau equation}
\begin{split}
\ddbar(e^{\vp}\omega & -\alpha e^{-\vp}\rho)\wedge\omega^{n-2} \\
                     & +n\alpha\ddbar\vp\wedge\ddbar\vp\wedge\omega^{n-2}+\mu\frac{\omega^{n}}{n!}=0,
\end{split}
\end{equation}
where $\alpha$ is a non-zero constant called the slope parameter, $\rho$ is a real smooth $(1,1)$ form, $\mu$ is a smooth function. For $\vp$, we impose the elliptic condition
\begin{equation}\label{Elliptic condition}
\tilde{\omega}=e^{\vp}\omega+\alpha e^{-\vp}\rho+2n\alpha\ddbar\vp \in \Gamma_{2}(M)
\end{equation}
and the normalization condition
\begin{equation}\label{Normalization condition}
\|e^{-\vp}\|_{L^{1}}=A,
\end{equation}
where
\begin{equation*}
\Gamma_{2}(M)=\{ \alpha\in A^{1,1}(M)~|~ \frac{\alpha^{1}\wedge\omega^{n-1}}{\omega^{n}}>0, \frac{\alpha^{2}\wedge\omega^{n-2}}{\omega^{n}}>0 \}
\end{equation*}
and $A^{1,1}(M)$ is the space of smooth real (1,1) forms on $M$.

When $n=2$, (\ref{Fu-Yau equation}) is equivalent to the Strominger system on a toric fibration over a $K3$ surface constructed by Goldstein and Prokushki \cite{GP04}, which was solved by Fu and Yau for $\alpha>0$ and $\alpha<0$ in \cite{FuY08} and \cite{FuY07}, respectively. (\ref{Fu-Yau equation}) is usually called Fu-Yau equation (cf. \cite{PPZ16b,GM16}).

In case of $\alpha<0$, Phong, Picard and Zhang \cite{PPZ16b} recently  proved the existence  of solutions of (\ref{Fu-Yau equation}) with the condition (\ref{Normalization condition}) is replaced by
\begin{equation*}
\|e^{\vp}\|_{L^{1}} = \frac{1}{A} \gg 1.
\end{equation*}
In \cite{CHZ18}, we prove that there exists constant $A_{0}$ depending only on $\alpha$, $\rho$, $\mu$ and $(M,\omega)$ such that for any $A\leq A_{0}$ and any $\alpha\neq0$, (\ref{Fu-Yau equation}) has a solution satisfying (\ref{Elliptic condition}) and  (\ref{Normalization condition}). \footnote{Phong, Pacard and Zhang  posted a paper \cite{PPZ18}  with a similar result after we posted the paper in arXiv.}  Our result is new and different from that of \cite{PPZ16b} which deals with only the case that $\alpha<0$.

Since the Strominger system comes from non-K\"{a}hler geometry \cite{Str86}, it is natural to consider (\ref{Fu-Yau equation})  on Hermitian manifolds. In the K\"{a}hler case, all the proofs in \cite{FuY08,FuY07,PPZ16b,CHZ18} relied heavily on the K\"{a}hler condition $d\omega=0$. It seems to be very difficult to solve (\ref{Fu-Yau equation}) on general Hermitian manifolds. In this paper,  we  focus on  a class of Hermitian manifolds which satisfies the astheno-K\"{a}hler condition
\begin{equation}\label{Astheno-Kahler condition}
\ddbar\omega^{n-2}=0.
\end{equation}
 Astheno-K\"{a}hler manifold  was first introduced in the paper of Jost-Yau \cite{JY93}, where they  extended  Siu's rigidity results  in  K\"{a}hler manifolds to astheno-K\"{a}hler manifolds  \cite{SY80}.   Such manifolds have many naturally properties as  K\"{a}hler manifolds.   For example, every holomorphic $1$-form on a compact astheno-K\"{a}hler manifold is closed \cite[Lemma 6]{JY93}.

There are many examples  of  astheno-K\"{a}hler manifolds, see \cite{LY87,LYZ94,MT01,Ma09,FT11,FGV16,LU17}.   For example, the product of a complex curve with a K\"{a}hler metric and a complex surface with a non-K\"{a}hler Gauduchon metric satisfies (\ref{Astheno-Kahler condition}).

The purpose of  this paper is to generalize the main result in  \cite{CHZ18} to  astheno-K\"{a}hler manifolds. Namely, we prove

\begin{theorem}\label{Existence and Uniqueness Theorem}
Let $(M,\omega)$ be an $n$-dimensional compact astheno-K\"{a}hler manifold.  Then there exist constants $A_{0}$, $C_{0}$, $\delta_{0}$, $M_{0}$ and $D_{0}$ depending only on $\alpha$, $\rho$, $\mu$ and $(M,\omega)$ such that for any $A\leq A_{0}$, there exists a unique solution $\vp$ of (\ref{Fu-Yau equation}) satisfying  (\ref{Elliptic condition}), (\ref{Normalization condition}) and
\begin{equation}\label{Restrictions}
e^{-\vp} \leq \delta_{0},~~|\de\dbar\vp|_{g} \leq D,~~D_{0}\leq D \text{~and~} A \leq \frac{1}{C_{0}M_{0}D}.
\end{equation}
\end{theorem}

  Since  in the non-K\"{a}hler case the constant  is not  a trivial solution   in  the  continuity method    when $t=0$  in \cite{CHZ18},     we introduce  a new  continuous   path to solve (\ref{Fu-Yau equation}) as follows ($t\in [0,1]$),
\begin{equation}\label{Fu-Yau equation t}
\begin{split}
\ddbar(e^{\vp}\omega & -t\alpha e^{-\vp}\rho)\wedge\omega^{n-2}+n\alpha\ddbar\vp\wedge\ddbar\vp\wedge\omega^{n-2} \\
& +n\alpha(t-1)\ddbar h\wedge\ddbar h\wedge\omega^{n-2}+t\mu\frac{\omega^{n}}{n!} = 0,
\end{split}
\end{equation}
where $h$ is a smooth function. Clearly,    (\ref{Fu-Yau equation t}) is equivalent  to  (\ref{Fu-Yau equation}) when $t=1.$
We will show that there is  $h$ such that
 (\ref{Fu-Yau equation t}) can be  solved   when $t=0$ (cf. Lemma \ref{Existence lemma}).

In the proof of openness for the solvable set of $t$, the astheno-K\"{a}hler condition (\ref{Astheno-Kahler condition}) will play an important role. (\ref{Astheno-Kahler condition}) guarantees that the adjoint  of linearized operator $L$ has no zero order terms (cf. (\ref{Definition of L^*})), then the strong maximum principle can be applied.

 (\ref{Astheno-Kahler condition})  will also be used  for  the $C^{0}$-estimate (cf. Lemma \ref{Zero order estimate lemma}). In fact,  instead  of $L^{1}$-integral  of $\varphi$ in \cite{CHZ18},   we first estimate a $L^{k_{0}}$-integral  for some $k_{0}\ll1$,  then apply the Moser iteration to derive the  $C^0$-estimate.  The  $C^1, C^2$-estimates for solutions of (\ref{Fu-Yau equation t}) can be obtained by  the argument in \cite{CHZ18}. For the reader's convenience, we give a sketch of the proofs  in Section 3.  Actually, the  argument there  are valid for solutions of (\ref{Fu-Yau equation}) on any Hermitian manifolds $(M,\omega)$.

In the next part of  this paper,   we improve Theorem \ref{Existence and Uniqueness Theorem} without the  restriction condition (\ref{Restrictions}) in case of $\alpha<0$.  In fact,   we prove  the  following uniqueness  of Fu-Yau equation.

\begin{theorem}\label{Uniqueness theorem alpha negative}
Let $\alpha<0$ and $(M,\omega)$ be an $n$-dimensional compact astheno-K\"{a}hler manifold. There exists a constant $A_{0}$ depending only on $\alpha$, $\rho$, $\mu$ and $(M,\omega)$ such that for any $A\leq A_{0}$, (\ref{Fu-Yau equation}) has a unique smooth solution satisfying (\ref{Elliptic condition}) and the $L^{n}$-normalization condition
\begin{equation}\label{L^n normalization condition}
\|e^{-\vp}\|_{L^{n}}=A.
\end{equation}
\end{theorem}

Furthermore, investigating the structure of the Fu-Yau equation, we obtain the monotonicity property of solutions.

\begin{theorem}\label{Monotonicity theorem}
Let $\alpha<0$ and $(M,\omega)$ be an $n$-dimensional compact K\"{a}hler manifold. Suppose that $\vp$ and $\ti{\vp}$ are solutions of (\ref{Fu-Yau equation}) satisfying (\ref{Elliptic condition}). If $\|e^{-\vp}\|_{L^{n}}=A$, $\|e^{-\ti{\vp}}\|_{L^{n}}=\ti{A}$ and $A<\ti{A}\leq A_{0}$, then we have $\vp>\ti{\vp}$ on $M$, where $A_{0}$ is a constant depending only on $\alpha$, $\rho$, $\mu$ and $(M,\omega)$.
\end{theorem}

\begin{remark}
In addition, if $\textrm{tr}_{\omega}\rho\geq0$, both Theorem \ref{Uniqueness theorem alpha negative} and \ref{Monotonicity theorem} are still true when $L^{n}$ normalization condition (\ref{L^n normalization condition}) is replaced by a weaker condition $\|e^{-\vp}\|_{L^{1}}=A$ (see Remark \ref{Special case remark}). In particular,  Theorem \ref{Uniqueness theorem alpha negative} is  an improvement of main results in  \cite{FuY07, PPZ16b} in K\"{a}hler case.
\end{remark}

Roughly speaking, Theorem \ref{Uniqueness theorem alpha negative} and \ref{Monotonicity theorem} are consequences of a priori estimates for $\varphi$.  Compared to the proof of   Theorem \ref{Existence and Uniqueness Theorem}, we need to derive a strong $C^0, C^1, C^2$ estimates without  (\ref{Restrictions}).
In order to use
 the blow-up argument for  $C^1, C^2$ estimates,  we  establish  an estimate
\begin{equation}\label{c1-c2-estimate}
\sup_{M}|\de\dbar\vp|_{g} \leq C_A(1+\sup_{M}|\de\vp|_{g}^{2}),
\end{equation}
where $C_{A}$ is a constant depending only on $A$, $\alpha$, $\rho$, $\mu$ and $(M,\omega)$.
Such a kind of estimate (\ref{c1-c2-estimate}) was widely studied in  Monge-Amp\`ere equations and  $\sigma_k$ Hessian equations (cf. \cite{WY09,HMW10,DK12,TW17,PPZ16b,Sze15,STW17}). In our case, we adopt an auxiliary function involving the largest eigenvalue $\lambda_{1}$  of  $\ti{\omega}$ with respect to
$\omega$. This advantage  gives us enough good third order terms to deal with the bad terms when we use the maximal principle as in \cite{CHZ17}.  Also  the sign of $\alpha$ plays a crucial role. We note that (\ref{Fu-Yau equation}) is not degenerate when $\alpha<0$ (cf. (\ref{Fu-Yau equation 2-nd Hessian type})).

 As we know, (\ref{Fu-Yau equation}) can be rewritten as a  $\sigma_{2}$-type  equation on a  Hermitian manifold with  function $F$  at the right hand  including the gradient term of solution (cf. (\ref{Hessian type of Fu-Yau equation}), (\ref{Fu-Yau equation 2-nd Hessian type})). In \cite{CHZ17}, we generalized  $\sigma_{2}$-equation to  an almost Hermitian manifold and   obtained  a $C^2$-estimate  for the  solutions, which depends only on  the gradient of solutions  and background data.
    It is interesting to studying  the  $C^2$-estimate  for solutions of $\sigma_{k}$-type equation in space of $\Gamma_k$ $(k\ge 2)$  of  $k$-convex functions  (cf. \cite{GRW15, Sze15, PPZ15, PPZ16a}, etc.).   But it seems nontrivial to generalize the method for    $\sigma_{2}$-equation to $\sigma_{k}$-equation even on  K\"{a}hler manifolds  if $F$  involves the gradient term of solution.

The organization of paper is as follows.  In Section 2 and Section 3,   we give the $C^0$-estimate, and $C^1, C^2$-estimates for  solutions of
(\ref{Fu-Yau equation}) under the condition in Theorem \ref{Existence and Uniqueness Theorem}, respectively.  Theorem \ref{Existence and Uniqueness Theorem} is proved in Section 4.  In Section 5, we improve the  $C^0$-estimate in Section 2 in case of $\alpha <0$. In Section 6, we give another
method to get strong $C^1, C^2$-estimates  in case of $\alpha <0$.   Theorem \ref{Uniqueness theorem alpha negative} and  Theorem \ref{Monotonicity theorem} will be proved in Section 7, 8, respectively.

\section{Zero order estimate (I)}
In this section, we use the Moser iteration to do  $C^{0}$-estimate for solutions  $\vp$ of (\ref{Fu-Yau equation}). First,  we prove a lemma for $L^{2}$-estimate of gradient $\de\vp$.

\begin{lemma}\label{Zero order estimate lemma}
Let $\vp$ be a smooth solution of (\ref{Fu-Yau equation}) satisfying (\ref{Elliptic condition}).  Let  $f(t)$ be a smooth  function in $\mathbb R^1$  such that $f'\geq0$. Then we have
\begin{align}\label{l2-gradient}
        & \int_{M}f'(\vp)\sqrt{-1}\de\vp\wedge\dbar\vp\wedge(e^{\vp}\omega+\alpha e^{-\vp}\rho)\wedge\omega^{n-2} \notag\\
&\leq   -2\int_{M}f'(\vp)\sqrt{-1}\de\vp\wedge(e^{\vp}\dbar\omega-\alpha e^{-\vp}\dbar\rho)\wedge\omega^{n-2}\notag \\
        & -2\int_{M}f'(\vp)\dbar\vp\wedge(e^{\vp}\omega-\alpha e^{-\vp}\rho)\wedge\sqrt{-1}\de\omega^{n-2}
          +2\int_{M}f(\vp)\mu\frac{\omega^{n}}{n!}.
\end{align}

\end{lemma}

\begin{proof}
Since $\ti{\omega}\in\Gamma_{2}(M)$ and $f'(\vp)\geq0$, it is clear that
\begin{equation*}
\int_{M}f'(\vp)\sqrt{-1}\de\vp\wedge\dbar\vp\wedge\ti{\omega}\wedge\omega^{n-2} \geq 0.
\end{equation*}
By the Stokes' formula , it follows
\begin{equation*}
\begin{split}
        & \int_{M}f'(\vp)\sqrt{-1}\de\vp\wedge\dbar\vp\wedge(e^{\vp}\omega+\alpha e^{-\vp}\rho)\wedge\omega^{n-2} \\
& \geq    -2n\alpha\int_{M}\sqrt{-1}\de f(\vp)\wedge\dbar\vp\wedge\ddbar\vp\wedge\omega^{n-2} \\
   &= -2n\alpha\int_{M}f(\vp)\dbar\vp\wedge\ddbar\vp\wedge\sqrt{-1}\de\omega^{n-2} \\
      &\ {} \ +2n\alpha\int_{M}f(\vp)\ddbar\vp\wedge\ddbar\vp\wedge\omega^{n-2}.
\end{split}
\end{equation*}
On the other hand, by (\ref{Fu-Yau equation}), we have
\begin{equation*}
\begin{split}
     & 2n\alpha\int_{M}f(\vp)\ddbar\vp\wedge\ddbar\vp\wedge\omega^{n-2} \\
= {} & -2\int_{M}f(\vp)\ddbar(e^{\vp}\omega-\alpha e^{-\vp}\rho)\wedge\omega^{n-2}-2\int_{M}f(\vp)\mu\frac{\omega^{n}}{n!}.
\end{split}
\end{equation*}
Thus
\begin{equation}\label{Zero order estimate equ 1}
\begin{split}
        & \int_{M}f'(\vp)\sqrt{-1}\de\vp\wedge\dbar\vp\wedge(e^{\vp}\omega+\alpha e^{-\vp}\rho)\wedge\omega^{n-2} \\
\geq {} & -2n\alpha\int_{M}f(\vp)\dbar\vp\wedge\ddbar\vp\wedge\sqrt{-1}\de\omega^{n-2} \\
        & -2\int_{M}f(\vp)\ddbar(e^{\vp}\omega-\alpha e^{-\vp}\rho)\wedge\omega^{n-2}-2\int_{M}f(\vp)\mu\frac{\omega^{n}}{n!}.
\end{split}
\end{equation}

For the first term on the right hand of the inequality (\ref{Zero order estimate equ 1}), we have
\begin{equation}\label{Zero order estimate equ 2}
\begin{split}
& -2n\alpha\int_{M}f(\vp)\dbar\vp\wedge\ddbar\vp\wedge\sqrt{-1}\de\omega^{n-2} \\
= {} & 2n\alpha\int_{M}\dbar(f(\vp)\dbar\vp)\wedge\sqrt{-1}\de\vp\wedge\sqrt{-1}\de\omega^{n-2} \\
     & +2n\alpha\int_{M}f(\vp)\dbar\vp\wedge\sqrt{-1}\de\vp\wedge\sqrt{-1}~\dbar\de\omega^{n-2} \\
= {} & 2n\alpha\int_{M}f(\vp)\sqrt{-1}\de\vp\wedge\dbar\vp\wedge\ddbar\omega^{n-2}.
\end{split}
\end{equation}
For the second term of (\ref{Zero order estimate equ 1}), we compute
\begin{equation*}
\begin{split}
     & -2\int_{M}f(\vp)\ddbar(e^{\vp}\omega-\alpha e^{-\vp}\rho)\wedge\omega^{n-2} \\
= {} & 2\int_{M}\sqrt{-1}\de f(\vp)\wedge\dbar(e^{\vp}\omega-\alpha e^{-\vp}\rho)\wedge\omega^{n-2} \\
     & -2\int_{M}f(\vp)\dbar(e^{\vp}\omega-\alpha e^{-\vp}\rho)\wedge\sqrt{-1}\de\omega^{n-2}.
\end{split}
\end{equation*}
and then,
\begin{equation}\label{Zero order estimate equ 3}
\begin{split}
     & -2\int_{M}f(\vp)\ddbar(e^{\vp}\omega-\alpha e^{-\vp}\rho)\wedge\omega^{n-2} \\
= {} & 2\int_{M}f'(\vp)\sqrt{-1}\de\vp\wedge\dbar\vp\wedge(e^{\vp}\omega+\alpha e^{-\vp}\rho)\wedge\omega^{n-2} \\
     & +2\int_{M}f'(\vp)\sqrt{-1}\de\vp\wedge(e^{\vp}\dbar\omega-\alpha e^{-\vp}\dbar\rho)\wedge\omega^{n-2} \\
     & +2\int_{M}\dbar f(\vp)\wedge(e^{\vp}\omega-\alpha e^{-\vp}\rho)\wedge\sqrt{-1}\de\omega^{n-2} \\
     & -2\int_{M}f(\vp)(e^{\vp}\omega-\alpha e^{-\vp}\rho)\wedge\ddbar\omega^{n-2}.
\end{split}
\end{equation}
Thus substituting (\ref{Zero order estimate equ 2}) and (\ref{Zero order estimate equ 3}) into (\ref{Zero order estimate equ 1}), we see that
\begin{equation*}
\begin{split}
        & \int_{M}f'(\vp)\sqrt{-1}\de\vp\wedge\dbar\vp\wedge(e^{\vp}\omega+\alpha e^{-\vp}\rho)\wedge\omega^{n-2} \\
\leq {} & -2\int_{M}f'(\vp)\sqrt{-1}\de\vp\wedge(e^{\vp}\dbar\omega-\alpha e^{-\vp}\dbar\rho)\wedge\omega^{n-2} \\
        & -2\int_{M}f'(\vp)\dbar\vp\wedge(e^{\vp}\omega-\alpha e^{-\vp}\rho)\wedge\sqrt{-1}\de\omega^{n-2} \\
        & +2\int_{M}f(\vp)(e^{\vp}\omega-\alpha e^{-\vp}\rho)\wedge\ddbar\omega^{n-2} \\
        & -2n\alpha\int_{M}f(\vp)\sqrt{-1}\de\vp\wedge\dbar\vp\wedge\ddbar\omega^{n-2}+2\int_{M}f(\vp)\mu\frac{\omega^{n}}{n!}.
\end{split}
\end{equation*}
Note that  $\ddbar\omega^{n-2}=0$. Hence,  we get (\ref{l2-gradient}).
\end{proof}

By Lemma \ref{Zero order estimate lemma}, we prove the following  $C^0$-estimate.

\begin{proposition}\label{Zero order estimate}
Let $\vp$ be a smooth solution of (\ref{Fu-Yau equation}) satisfying (\ref{Elliptic condition}) and (\ref{Normalization condition}). There exist constants $\delta_{0}$, $A_{0}$ and $M_{0}$ depending only on $\alpha$, $\rho$, $\mu$ and $(M,\omega)$ such that if
\begin{equation}\label{two-condition}
e^{-\vp}\leq\delta_{0} ~\text{~and~}~ \|e^{-\vp}\|_{L^{1}} = A \leq A_{0},
\end{equation}
then
\begin{equation*}
\frac{A}{M_{0}} \leq e^{-\vp} \leq M_{0}A.
\end{equation*}
\end{proposition}

\begin{proof}
First, we estimate the positive infimum of $e^{\vp}$. At the expense of decreasing $\delta_{0}$, we assume that
\begin{equation}\label{Infimum estimate equ 1}
e^{\vp}\omega+\alpha e^{-\vp}\rho \geq \frac{1}{2}e^{\vp}\omega.
\end{equation}
Then  by taking $f(\vp)=-e^{-(k+1)\vp}$  ($k\geq1$) in Lemma \ref{Zero order estimate lemma},  we have
\begin{equation*}
\begin{split}
\int_{M}e^{-k\vp}|\de\vp|_{g}^{2}\omega^{n}
\leq C\int_{M}e^{-k\vp}|\de\vp|_{g}\omega^{n}+C\int_{M}e^{-(k+1)\vp}\omega^{n}.
\end{split}
\end{equation*}
By the Cauchy-Schwarz inequality, it follows that
\begin{equation}\label{e-phi}
\int_{M}e^{-k\vp}|\de\vp|_{g}^{2}\omega^{n} \leq C\int_{M}e^{-k\vp}\omega^{n}.
\end{equation}
Hence, by the above relation together with the Sobolev inequality,  one can use  the Moser iteration to derive
\begin{equation}\label{Infimum estimate equ 2}
\|e^{-\vp}\|_{L^{\infty}} \leq C\|e^{-\vp}\|_{L^{1}} = CA.
\end{equation}

Next we  estimate the supremum of $e^{\vp}$. As in the proof of (\ref{e-phi}), by  taking
\begin{equation*}
f(\vp) = \frac{1}{k-1}e^{(k-1)\vp}
\end{equation*}
in Lemma \ref{Zero order estimate lemma}, we can also get
\begin{equation}\label{Supremum estimate equ 1}
\begin{split}
\int_{M}e^{k\vp}|\de\vp|_{g}^{2}\omega^{n} \leq C\left(1+\frac{1}{|k-1|}\right)\int_{M}e^{k\vp}\omega^{n}.
\end{split}
\end{equation}

\begin{claim}\label{Claim} There exists a positive  constant $k_{0}\ll1$ depending only on $\alpha$, $\rho$, $\mu$ and $(M,\omega)$ such that
\begin{equation}\label{L^k0 estimate}
\|e^{\vp}\|_{L^{k_{0}}} \leq \frac{C}{A}.
\end{equation}
\end{claim}

By  (\ref{Supremum estimate equ 1}),  we  use  the Moser iteration to obtain
\begin{equation*}
\|e^{\vp}\|_{L^{\infty}} \le C_{0} \|e^{\vp}\|_{L^2}.
\end{equation*}
By (\ref{L^k0 estimate}),  it follows
 $$  \|e^{\vp}\|_{L^{\infty}} \leq C_{0}^{\frac{2}{k_{0}}} \|e^{\vp}\|_{L^{k_{0}}} \leq \frac{C}{A}.$$
Thus, the proof of Proposition  \ref{Zero order estimate} is  complete.

It remains to prove  Claim \ref{Claim}. Without loss of generality, we assume that $\vol(M,\omega)=1$. We define
\begin{equation*}
U = \{x\in M~|~e^{-\vp(x)}\geq \frac{A}{2} \}.
\end{equation*}
Then by  (\ref{Infimum estimate equ 2}),  we have
\begin{equation*}
\begin{split}
A =    {} & \int_{M}e^{-\vp}\omega^{n} \\
            =    {} & \int_{U}e^{-\vp}\omega^{n}+\int_{M\setminus U}e^{-\vp}\omega^{n} \\
            \leq {} & e^{-\inf_{M}\vp}\vol(U)+\frac{A}{2}(1-\vol(U)) \\
            \leq {} & \left(C-\frac{1}{2}\right)A\vol(U)+\frac{A}{2},
\end{split}
\end{equation*}
which implies
\begin{equation}\label{Supremum estimate equ 2}
\vol(U) \geq \frac{1}{C_{0}}.
\end{equation}
On the other hand, by the Poincar\'{e} inequality, we have
\begin{equation*}
\int_{M}e^{k_{0}\vp}\omega^{n}-\left(\int_{M}e^{\frac{k_{0}\vp}{2}}\omega^{n}\right)^{2}
\leq C\int_{M}|\de e^{\frac{k_{0}\vp}{2}}|_{g}^{2}\omega^{n}
\leq Ck_{0}^{2}\int_{M}e^{k_{0}\vp}\omega^{n}.
\end{equation*}
It then follows that
\begin{equation}\label{Supremum estimate equ 3}
\int_{M}e^{k_{0}\vp}\omega^{n}
\leq \frac{1}{1-C_{0}k_{0}^{2}}\left(\int_{M}e^{\frac{k_{0}\vp}{2}}\omega^{n}\right)^{2}.
\end{equation}
Combining this with the Cauchy-Schwarz inequality, we obtain
\begin{equation*}
\begin{split}
& \left(\int_{M}e^{\frac{k_{0}\vp}{2}}\omega^{n}\right)^{2} \\
\leq {} & (1+C_{0})\left(\int_{U}e^{\frac{k_{0}\vp}{2}}\omega^{n}\right)^{2}
          +\left(1+\frac{1}{C_{0}}\right)\left(\int_{M\setminus U}e^{\frac{k_{0}\vp}{2}}\omega^{n}\right)^{2}\\
\leq {} & \frac{(1+C_{0})2^{k_{0}}}{A^{k_{0}}}(\vol(U))^{2}+\left(1+\frac{1}{C_{0}}\right)(1-\vol(U))^{2}\int_{M}e^{k_{0}\vp}\omega^{n}\\
\leq {} & \frac{(1+C_{0})2^{k_{0}}}{A^{k_{0}}}
          +\left(1-\frac{1}{C_{0}^{2}}\right)\frac{1}{1-C_{0}k_{0}^{2}}\left(\int_{M}e^{\frac{k_{0}\vp}{2}}\omega^{n}\right)^{2}.
\end{split}
\end{equation*}
By choosing  $k_{0}\ll 1$,  we see that
\begin{equation*}
\left(\int_{M}e^{\frac{k_{0}\vp}{2}}\omega^{n}\right)^{2} \leq \frac{C}{A^{k_{0}}}.
\end{equation*}
Thus,  we get from (\ref{Supremum estimate equ 3}),
\begin{equation*}
\int_{M}e^{k_{0}\vp}\omega^{n} \leq \frac{C}{A^{k_{0}}}.
\end{equation*}
 Claim \ref{Claim} is proved.
\end{proof}

\section{First and second order estimates (I)}

In this section, we give a sketch of proofs of  $C^1, C^2$ estimates of $\varphi$.  As in \cite {CHZ18}, the basic idea is to  rewrite (\ref{Fu-Yau equation}) as a $\sigma_2$-type  equation,
\begin{equation}\label{Hessian type of Fu-Yau equation}
\sigma_{2}(\ti{\omega}) = \frac{n(n-1)}{2}\left(e^{2\vp}-4\alpha e^\varphi|\de\vp|_{g}^{2}\right)+\frac{n(n-1)}{2}f,
\end{equation}
where
\begin{equation}\label{Definition of f}
\begin{split}
  & \ f\omega^{n} \\
= & \ 2\alpha\rho\wedge\omega^{n-1}+\alpha^{2}e^{-2\vp}\rho^{2}\wedge\omega^{n-2}-4n\alpha\mu\frac{\omega^{n}}{n!} \\
+ & \ 4n\alpha^{2}e^{-\vp}\sqrt{-1}\left(\de\vp\wedge\dbar\vp\wedge\rho-\de\vp\wedge\dbar\rho
       -\de\rho\wedge\dbar\vp+\de\dbar\rho\right)\wedge\omega^{n-2} \\[1mm]
+ & \ 4n\alpha e^{\vp}\sqrt{-1}\left(\de\omega\wedge\dbar\vp+\de\vp\wedge\dbar\omega+\de\dbar\omega\right)\wedge\omega^{n-2}.
\end{split}
\end{equation}
As in \cite{CHZ18}, we define $\hat{\omega}=e^{-\vp}\ti{\omega}$. Then (\ref{Hessian type of Fu-Yau equation}) becomes
\begin{equation}\label{New Hessian type of Fu-Yau equation}
\sigma_{2}(\hat{\omega}) = \frac{n(n-1)}{2}\left(1-4\alpha e^{-\vp}|\de\vp|_{g}^{2}\right)+\frac{n(n-1)}{2}e^{-2\vp}f.
\end{equation}
Since $\omega$ is  not K\"{a}hler,    the function $f$  is  more complicated than one in \cite{CHZ18}. Precisely, more terms involving $e^{\vp}$ and $\de\vp$ appears. However, for the right hand side of (\ref{New Hessian type of Fu-Yau equation}), the leading term is still $-2n(n-1)\alpha e^{-\vp}|\de\vp|_{g}^{2}$. Thus  we will obtain a similar inequality as in K\"{a}hler case when we differentiate (\ref{New Hessian type of Fu-Yau equation}). This is  why we  can  prove  an analogy of \cite[Propsition 3.1, 4.1]{CHZ18} as follows.

\begin{proposition}\label{First and second order estimate}
Let $\vp$ be a smooth solution of (\ref{Fu-Yau equation}) satisfying (\ref{Elliptic condition}) and $\frac{1}{M_{0}A}\leq e^{-\vp}\leq M_{0}A$ for some uniform constant $M_0$. There exist uniform constants $D_{0}$ and $C_{0}$ such that if
\begin{equation}\label{condition of first and second order estimate}
|\de\dbar\vp|_{g} \leq D, ~~D_{0}\leq D \text{~and~} A\leq A_{D}:=\frac{1}{C_{0}M_{0}D},
\end{equation}
then
\begin{equation*}
|\de\vp|_{g}^{2}\leq M_{1} \text{~and~} |\de\dbar\vp|_{g} \leq \frac{D}{2}.
\end{equation*}
\end{proposition}

\begin{proof}

i).  $C^1$-estimate.   As in \cite{CHZ18},  we consider  the following  auxiliary function,
\begin{equation*}
Q = \log|\de\vp|_{g}^{2}+\frac{\vp}{B},
\end{equation*}
where $B$ is a uniform constant to be determined. Let $x_{0}$ be the maximum point of $Q$ and $\{e_{i}\}_{i=1}^{n}$ be a local unitary frame in a neighbourhood of $x_{0}$ such that, at $x_{0}$,
\begin{equation}\label{tilde gij}
\tilde{g}_{i\ol{j}}
= \delta_{i\ol{j}}\tilde{g}_{i\ol{i}}
= \delta_{i\ol{j}}(e^{\vp}+\alpha e^{-\vp}\rho_{i\ov{i}}+2n\alpha \vp_{i\ov{i}}).
\end{equation}
We use the following notations
\begin{equation*}
F^{i\ov{j}} = \frac{\de\sigma_{2}(\hat{\omega})}{\de\hat{g}_{i\ov{j}}}
\text{~and~}
F^{i\ov{j},k\ov{l}} = \frac{\de^{2}\sigma_{2}(\hat{\omega})}{\de\hat{g}_{i\ov{j}}\de\hat{g}_{k\ov{l}}},
\end{equation*}
where $\hat{\omega}=e^{-\vp}\ti{\omega}$.
By (\ref{condition of first and second order estimate}), we know that
\begin{equation}\label{Bound of Fij}
\left|F^{i\ov{i}}-(n-1)\right| \leq \frac{1}{100}.
\end{equation}
Then by a direct calculation, we have
\begin{equation}\label{Maximum principle gradient 1}
\begin{split}
F^{i\ov{j}}e_{i}\ov{e}_{j}(|\de\vp|_{g}^{2})
\geq {} & \frac{4}{5}\sum_{i,j}(|e_{i}\ov{e}_{j}(\vp)|^{2}+|e_{i}e_{j}(\vp)|^{2})-C|\de\vp|_{g}^{2} \\
        & +\sum_{k}F^{i\ov{i}}\left(e_{k}(\vp_{i\ov{i}})\vp_{\ov{k}}+\ov{e}_{k}(\vp_{i\ov{i}})\vp_{k}\right).
\end{split}
\end{equation}

Next, we deal with the terms involving three derivatives of $\vp$ in \eqref{Maximum principle gradient 1}.
Differentiating    \eqref{New Hessian type of Fu-Yau equation}  along $e_{k}$ at $x_{0}$, we  have
\begin{equation}\label{first differentiate of equation}
\begin{split}
2n\alpha F^{i\ov{i}}e_{k}(\vp_{i\ov{i}})
= {} & 2\alpha e^{-\vp}\vp_{k}F^{i\ov{i}}\rho_{i\ov{i}}-\alpha e^{-\vp}F^{i\ov{i}}e_{k}(\rho_{i\ov{i}})+2n\alpha\vp_{k}F^{i\ov{i}}\vp_{i\ov{i}} \\[1mm]
     & -2\alpha n(n-1)(|\de\vp|_{g}^{2})_{k}+2\alpha n(n-1)|\de\vp|_{g}^{2}\vp_{k} \\
     & -n(n-1)e^{-\vp}f\vp_{k}+\frac{n(n-1)}{2}e^{-\vp}f_{k},
\end{split}
\end{equation}
which implies
\begin{equation}\label{Gradient estimate equation 7}
\begin{split}
        & \sum_{k}F^{i\ov{i}}\left(e_{k}(\vp_{i\ov{i}})\vp_{\ov{k}}+\ov{e}_{k}(\vp_{i\ov{i}})\vp_{k}\right) \\[1mm]
\geq {} & -Ce^{-\vp}|\de\vp|_{g}^{2}-Ce^{-\vp}|\de\vp|_{g}+2|\de\vp|_{g}^{2}F^{i\ov{i}}\vp_{i\ov{i}} \\[1mm]
        & +2(n-1)|\de\vp|_{g}^{4}
          -2(n-1)\textrm{Re}\left(\sum_{k}(|\de\vp|_{g}^{2})_{k}\vp_{\ov{k}}\right) \\
        & -\frac{n-1}{\alpha}e^{-\vp}|\de\vp|_{g}^{2}f
          +\frac{n-1}{2\alpha}e^{-\vp}\textrm{Re}\left(\sum_{k}f_{k}\vp_{\ov{k}}\right).
\end{split}
\end{equation}
For the third and fourth term of (\ref{Gradient estimate equation 7}), by the argument of \cite[(3.14)]{CHZ18}, we obtain
\begin{equation}\label{Gradient estimate equation 10}
\begin{split}
& 2|\de\vp|_{g}^{2}F^{i\ov{i}}\vp_{i\ov{i}}+2(n-1)|\de\vp|_{g}^{4} \\[2mm]
\geq {} & -\frac{1}{10}\sum_{i,j}|e_{i}\ov{e}_{j}(\vp)|^{2}-\left(Ce^{-2\vp}+\frac{1}{B}\right)|\de\vp|_{g}^{4}-C|\de\vp|_{g}^{2}.
\end{split}
\end{equation}
For the last two terms of (\ref{Gradient estimate equation 7}). By the similar calculation of \cite[(3.10)]{CHZ18} and the expression of $f$ (\ref{Definition of f}), at $x_{0}$, we get
\begin{equation}\label{the first order of f}
\begin{split}
& -\frac{n-1}{\alpha}e^{-\vp}|\de\vp|_{g}^{2}f+\frac{n-1}{2\alpha}e^{-\vp}\textrm{Re}\left(\sum_{k}f_{k}\vp_{\ov{k}}\right) \\[1mm]
\geq {} & -C\left(e^{-2\vp}|\de\vp|_{g}+e^{-2\vp}+|\de\vp|_{g}\right)\sum_{i,j}(|e_{i}\ov{e}_{j}(\vp)|+|e_{i}e_{j}(\vp)|) \\
        & -Ce^{-\vp}|\de\vp|_{g}^{4}-Ce^{-\vp}|\de\vp|_{g}^{3}-Ce^{-\vp}|\de\vp|_{g}^{2}-Ce^{-\vp}|\de\vp|_{g} \\[2mm]
        & -C|\de\vp|^{3}_{g}-C|\de\vp|^{2}_{g} \\[1mm]
\geq {} & -\frac{1}{10}\sum_{i,j}(|e_{i}\ov{e}_{j}(\vp)|^{2}+|e_{i}e_{j}(\vp)|^{2})-\left(Ce^{-\vp}+\frac{1}{2B}\right)|\de\vp|_{g}^{4} \\
        & -CB^{3},
\end{split}
\end{equation}
where we used the Cauchy-Schwarz inequality in the last inequality. Substituting (\ref{Gradient estimate equation 10}) and  (\ref{the first order of f}) into (\ref{Gradient estimate equation 7}), we derive
\begin{equation}\label{third term of gradient-lapalce}
\begin{split}
        & \sum_{k}F^{i\ov{i}}\left(e_{k}(\vp_{i\ov{i}})\vp_{\ov{k}}+\ov{e}_{k}(\vp_{i\ov{i}})\vp_{k}\right) \\
\geq {} & -\frac{1}{5}\sum_{i,j}(|e_{i}\ov{e}_{j}(\vp)|^{2}+|e_{i}e_{j}(\vp)|^{2})
          -2(n-1){\rm Re}\left(\sum_{k}(|\de\vp|_{g}^{2})_{k}\vp_{\ov{k}}\right) \\
        & -\left(Ce^{-\vp}+\frac{3}{2B}\right)|\de\vp|_{g}^{4}-C|\de\vp|_{g}^{2}-CB^{3}.
\end{split}
\end{equation}
Hence,  substituting this into \eqref{Maximum principle gradient 1}, we see that
\begin{equation}\label{Maximum principle gradient}
\begin{split}
& F^{i\ov{j}}e_{i}\ov{e}_{j}(|\de\vp|_{g}^{2}) \\
\geq {} & \frac{3}{5}\sum_{i,j}(|e_{i}\ov{e}_{j}(\vp)|^{2}+|e_{i}e_{j}(\vp)|^{2})
          -2(n-1)\textrm{Re}\left(\sum_{k}(|\de\vp|_{g}^{2})_{k}\vp_{\ov{k}}\right) \\
        & -\left(Ce^{-\vp}+\frac{3}{2B}\right)|\de\vp|_{g}^{4}-C|\de\vp|_{g}^{2}-CB^{3}.
\end{split}
\end{equation}
By the maximum principle, at $x_{0}$, we obtain
\begin{equation}\label{Gradient estimate equation 3}
\begin{split}
0 \geq {} & F^{i\ov{j}}e_{i}\ov{e}_{j}(Q) \\
  \geq {} & \frac{1}{2|\de\vp|_{g}^{2}}\sum_{i,j}(|e_{i}\ov{e}_{j}(\vp)|^{2}+|e_{i}e_{j}(\vp)|^{2})
            -\frac{2(n-1)\textrm{Re}\left(\sum_{k}(|\de\vp|_{g}^{2})_{k}\vp_{\ov{k}}\right)}{|\de\vp|_{g}^{2}} \\
          & -\frac{F^{i\ov{i}}|e_{i}(|\de\vp|_{g}^{2})|^{2}}{|\de\vp|_{g}^{4}}
            -\left(Ce^{-\vp}+\frac{1}{B}\right)|\de\vp|_{g}^{2}-C+\frac{1}{B}F^{i\ov{i}}e_{i}\ov{e}_{i}(\vp).
\end{split}
\end{equation}
On the other hand, by the fact $dQ=0$ and the Cauchy-Schwarz inequality, we get (cf. \cite[(3.17)-(3.19)]{CHZ18})
\begin{equation*}
\begin{split}
        & -\frac{2(n-1)\textrm{Re}\left(\sum_{k}(|\de\vp|_{g}^{2})_{k}\vp_{\ov{k}}\right)}{|\de\vp|_{g}^{2}}
          -\frac{F^{i\ov{i}}|e_{i}(|\de\vp|_{g}^{2})|^{2}}{|\de\vp|_{g}^{4}}
          +\frac{1}{B}F^{i\ov{i}}e_{i}\ov{e}_{i}(\vp) \\
\geq {} & -\frac{1}{4|\de\vp|_{g}^{2}}\sum_{i,j}|e_{i}\ov{e}_{j}(\vp)|^{2}+
          \frac{2(n-1)}{B}|\de\vp|_{g}^{2}-\frac{C}{B^{2}}|\de\vp|_{g}^{2}.
\end{split}
\end{equation*}
Substituting this into (\ref{Gradient estimate equation 3}), we see that
\begin{equation}\label{gradient lalpace}
\begin{split}
0 \geq {} & \frac{1}{4|\de\vp|_{g}^{2}}\sum_{i,j}\left(|e_{i}\ov{e}_{j}(\vp)|^{2}+|e_{i}e_{j}(\vp)|^{2}\right)-C_{0}B^{3} \\
          & +\left(\frac{4n-7}{2B}-\frac{C_{0}}{B^{2}}-C_{0}e^{-\vp}\right)|\de\vp|_{g}^{2}.
\end{split}
\end{equation}
Since $A\ll1$, we may assume
\begin{equation*}
C_{0}e^{-\vp} \leq \frac{1}{32C_{0}}.
\end{equation*}
By choosing  $B = 4C_{0}$ in (\ref{gradient lalpace}),   we see that
\begin{equation*}
|\de\vp|_{g}^{2}(x_{0}) \leq 2^{11}C_{0}^{5}.
\end{equation*}
Note that $\frac{1}{M_{0}A}\leq e^{-\vp}\leq M_{0}A$.  Hence,  we obtain
\begin{equation}\label{the first estimate}
\max_{M}|\de\vp|_{g}^{2} \leq e^{\frac{1}{B}(\sup_{M}\vp-\inf_{M}\vp)}|\de\vp|_{g}^{2}(x_{0}) \leq C.
\end{equation}

ii).  $C^2$-estimate.  The proof  is almost as  same as \cite[Propsition 4.1]{CHZ18}. We consider the following auxiliary function,
\begin{equation*}
Q = |\de\dbar\vp|_{g}^{2} + B|\de\vp|_{g}^{2},
\end{equation*}
where $B$ is a uniform constant to be determined.
Let $x_{0}$ be the maximum point of $Q$ and $\{e_{i}\}_{i=1}^{n}$ be the local unitary frame such that $\ti{g}(x_{0})$ is diagonal. By direct calculation, we have
\begin{equation}\label{Second order estimate equation 1}
\begin{split}
F^{i\ov{i}}e_{i}\ov{e}_{i}(|\de\dbar\vp|_{g}^{2})
   = {} & 2 \sum_{k,l}F^{i\ov{i}}e_{i}\ov{e}_{i}(\vp_{k\ov{l}})\vp_{l\ov{k}}+2\sum_{k,l}F^{i\ov{i}}e_{i}(\vp_{k\ov{l}})\ov{e}_{i}(\vp_{l\ov{k}}) \\
\geq {} & -2|\de\dbar\vp|_{g}\sum_{k,l}|F^{i\ov{i}}e_{i}\ov{e}_{i}(\vp_{k\ov{l}})|+\frac{1}{2}\sum_{i,j,p}|e_{p}e_{i}\ov{e}_{j}(\vp)|^{2} \\
        & -C\sum_{i,j}(|e_{i}\ov{e}_{j}(\vp)|^{2}+|e_{i}e_{j}(\vp)|^{2})-C.
\end{split}
\end{equation}
To deal with the fourth order terms  $ \sum_{k,l}F^{i\ov{i}}e_{i}\ov{e}_{i}(\vp_{k\ov{l}})\vp_{l\ov{k}}$ in  (\ref{Second order estimate equation 1}),   we  differentiate  (\ref{New Hessian type of Fu-Yau equation}) twice along $e_{k}$ and $\ov{e}_{l}$, we get
\begin{equation*}
\begin{split}
&F^{i\ov{j},p\ov{q}}e_{k}(e^{-\vp}\ti{g}_{i\ov{j}})\ov{e}_{l}  (e^{-\vp}\ti{g}_{p\ov{q}})
+F^{i\ov{j}}e_{k}\ov{e}_{l}(e^{-\vp}\ti{g}_{i\ov{j}}) \\
&= -2n(n-1)\alpha e_{k}\ov{e}_{l}(e^{-\vp}|\de\vp|_{g}^{2})+\frac{n(n-1)}{2}e_{k}\ov{e}_{l}(e^{-2\vp}f).
\end{split}
\end{equation*}
By the similar argument of \cite[Lemma 4.2]{CHZ18} and the expression of $f$ (\ref{Definition of f}), we obtain
\begin{equation*}
\begin{split}
|F^{i\ov{i}}e_{i}\ov{e}_{i}(\vp_{k\ov{l}})|
\leq {} & 8n|\alpha| e^{-\vp}\sum_{i,j,p}|e_{p}e_{i}\ov{e}_{j}(\vp)|^{2}+C\sum_{i,j,p}|e_{p}e_{i}\ov{e}_{j}(\vp)| \\
        & +C\sum_{i,j}(|e_{i}\ov{e}_{j}(\vp)|^{2}+|e_{i}e_{j}(\vp)|^{2})+C.
\end{split}
\end{equation*}
Substituting this into (\ref{Second order estimate equation 1}) and using \eqref{condition of first and second order estimate}, we obtain
\begin{equation*}
\begin{split}
F^{i\ov{i}}e_{i}\ov{e}_{i}(|\de\dbar\vp|_{g}^{2})
\geq -C_{0}(D+1)\sum_{i,j}(|e_{i}\ov{e}_{j}(\vp)|^{2}+|e_{i}{e}_{j}(\vp)|^{2})-C_{0}.
\end{split}
\end{equation*}
On the other hand, by \eqref{Maximum principle gradient} and $C^{1}$-estimate, we have
\begin{equation*}
F^{i\ov{i}}e_{i}\ov{e}_{i}(|\de\vp|_{g}^{2})
\geq \frac{1}{2}\sum_{i,j}(|e_{i}\ov{e}_{j}(\vp)|^{2}+|e_{i}{e}_{j}(\vp)|^{2})-C_{1}.
\end{equation*}
Hence, by the maximum principle, at $x_{0}$, we get
\begin{equation*}
\begin{split}
0 \geq {} & F^{i\ov{i}}e_{i}\ov{e}_{i}(Q) \\[1mm]
     = {} & F^{i\ov{i}}e_{i}\ov{e}_{i}(|\de\dbar\vp|_{g}^{2})+BF^{i\ov{i}}e_{i}\ov{e}_{i}(|\de\vp|_{g}^{2}) \\
  \geq {} & \left(\frac{B}{2}-C_{0}D-C_{0}\right)\sum_{i,j}(|e_{i}\ov{e}_{j}(\vp)|^{2}+|e_{i}{e}_{j}(\vp)|^{2})-C_{0}-C_{1}B.
\end{split}
\end{equation*}
Choose  $B=8C_{0}D+8C_{0}$.   It follows that
\begin{equation*}
|\de\dbar\vp|_{g}^{2}(x_{0}) \leq C.
\end{equation*}
Therefore,    by  \eqref{the first estimate}, at the expense of increasing $D_{0}$, we obtain
\begin{equation*}
\max_{M}|\de\dbar\vp|_{g}^{2} \leq |\de\dbar\vp|_{g}^{2}(x_{0})+BC \leq CD \leq \frac{D^{2}}{4}.
\end{equation*}
\end{proof}

\section{Proof of Theorem \ref{Existence and Uniqueness Theorem}}\label{Proof of existence theorem}

In this section, we solve (\ref{Fu-Yau equation t})  when  any $t\in [0,1]$.    The following lemma shows the existence of solutions  when $t=0$.

\begin{lemma}\label{Existence lemma}
Let $(M,\omega)$ be an $n$-dimensional compact astheno-K\"{a}hler manifold. Then there exists a function $h\in C^{\infty}(M)$, unique up to addition of a constant, such that
\begin{equation}\label{h-solution}
\ddbar(e^{h}\omega)\wedge\omega^{n-2} = 0.
\end{equation}
\end{lemma}

\begin{proof}
First we prove the existence. We define an elliptic operator $\ti{L}$ by
\begin{equation*}
(\ti{L}u)\omega^{n} = \ddbar(u\omega)\wedge\omega^{n-2}.
\end{equation*}
Let $\ti{L}^{*}$ be a $L^{2}$-adjoint   operator of $\ti{L}$.  Then by Stokes' formula and the condition $\ddbar\omega^{n-2}=0$, we see that
\begin{equation*}
(\ti{L}^{*}v)\omega^{n}
= \ddbar v\wedge\omega^{n-1}+\sqrt{-1}\de v\wedge\omega\wedge\dbar\omega^{n-2}-\dbar v\wedge\omega\wedge\sqrt{-1}\de\omega^{n-2}.
\end{equation*}
It follows that  ${\rm Ker}\ti{L}^{*}=\{\text{constants}\}$  and   $\textrm{Ind}(\ti{L})=\textrm{Ind}(\ti{L}^{*})=0$. Thus
\begin{equation*}
\textrm{dim}(\textrm{Ker}\ti{L}) = 1.
\end{equation*}
Denote     the generator of $\textrm{Ker}\ti{L}$  by  $v_{0}$. Then $v_{0}$ does not change the  sign,  and  we may  assume that $v_{0}\geq0$.  By the strong maximum principle,   we know that  $v_{0}>0$.   Hence,  $h=\log v_{0}$ satisfies (\ref{h-solution}).

For the uniqueness, by $\textrm{dim}(\textrm{Ker}\ti{L})=1$, we see that the only solution of (\ref{h-solution}) is $h+c$, where $c$ is a constant.
\end{proof}

Choose  the function $h$ in   (\ref{Fu-Yau equation t})  as a solution  of  (\ref{h-solution}).  We consider solution $\varphi=\varphi_t$ of  (\ref{Fu-Yau equation t})  which  satisfies the elliptic condition
\begin{equation}\label{Elliptic condition t}
e^{\vp}\omega+t\alpha e^{-\vp}\rho+2n\alpha\ddbar\vp \in \Gamma_{2}(M)
\end{equation}
and the normalization condition
\begin{equation}\label{Normalization condition t}
\|e^{-\vp}\|_{L^{1}}=A.
\end{equation}
By taking  $\vp_{0}=h+\log\|e^{-h}\|_{L^{1}}-\log A$,  it is easy to see that  $\|e^{-\vp_{0}}\|_{L^{1}}=A$ and  $\vp_{0}$ is a solution of   (\ref{Fu-Yau equation t})   when $t=0$.  Moreover, since  $A\ll 1$,
\begin{equation*}
e^{\vp_{0}}\omega+2n\alpha\ddbar\vp_{0}=\frac{\|e^{-h}\|_{L^{1}}}{A}e^{h}\omega+2n\alpha\ddbar h > 0.
\end{equation*}
Thus $\vp_{0}$ satisfies (\ref{Elliptic condition t}) and (\ref{Normalization condition t}).

For a fixed $\beta\in (0,1)$, we define
\begin{equation*}
\begin{split}
B & = \{ \vp\in C^{2,\beta}(M) ~|~ \|e^{-\vp}\|_{L^{1}} = A \},\\[2mm]
B_{1} & = \{ (\vp,t)\in B\times[0,1] ~|~ \text{$\vp$ satisfies (\ref{Elliptic condition t})} \},\\
B_{2} & = \{ u\in C^{\beta}(M) ~|~ \int_{M}u\omega^{n}=0 \}.
\end{split}
\end{equation*}
Then $B_{1}$ is an open subset of $B\times[0,1]$. Since $\int_{M}\mu\omega^{n}=0$, we introduce a map $\Phi:B_{1}\rightarrow B_{2}$,
\begin{equation*}
\begin{split}
\Phi(\vp,t)\omega^{n} = {} &\ddbar(e^{\vp}\omega-t\alpha e^{-\vp}\rho)\wedge\omega^{n-2} \\[1mm]
& +n\alpha\ddbar\vp\wedge\ddbar\vp\wedge\omega^{n-2} \\
& +2n\alpha(t-1)\ddbar h\wedge\ddbar h\wedge\omega^{n-2}+t\mu\frac{\omega^{n}}{n!}.
\end{split}
\end{equation*}
Set
\begin{equation*}
I=\{ t\in [0,1] ~|~ \text{there exists $(\vp,t)\in B_{1}$ such that $\Phi(\vp,t)=0$} \}.
\end{equation*}
Then the existence of  solutions of  (\ref{Fu-Yau equation}) is reduced to proving that $I$ is both open and closed.

\begin{proof}[Proof of  Theorem \ref{Existence and Uniqueness Theorem}]
\textbf{Openness.} Suppose that $\hat{t}\in I$ and  there exists $(\hat{\vp},\hat{t})\in B_{1}$ such that $\Phi(\hat{\vp},\hat{t})=0$. Let
\begin{equation*}
L: \{ u\in C^{2,\beta}(M) ~|~ \int_{M}ue^{-\hat{\vp}}\omega^{n}=0 \}
\rightarrow \{ v\in C^{\beta}(M) ~|~ \int_{M}v\omega^{n}=0 \}
\end{equation*}
be a  linearized operator of $\Phi$ at $\hat{\vp}$. Then
\begin{equation*}
\begin{split}
(Lu)\omega^{n} = {} & \ddbar(ue^{\hat{\vp}}\omega+\hat{t}\alpha ue^{-\hat{\vp}}\rho)\wedge\omega^{n-2} \\
                    & +2n\alpha\ddbar\hat{\vp}\wedge\ddbar u\wedge\omega^{n-2}.
\end{split}
\end{equation*}
By the implicit function theorem, it suffices to prove that $L$ is injective and surjective.

Let $L^{*}$ be a  $L^{2}$-adjoint  operaor  of $L$.  By the fact  $\ddbar\omega^{n-2}=0$ and Stokes' formula,  it follows that
\begin{equation}\label{Definition of L^*}
\begin{split}
L^{*}(v)\omega^{n}
= {} & \ddbar v\wedge\left(e^{\hat{\vp}}\omega+\hat{t}\alpha e^{-\hat{\vp}}\rho+2n\alpha\ddbar\hat{\vp}\right)\wedge\omega^{n-2} \\
     & +\sqrt{-1}\de v\wedge\left(e^{\hat{\vp}}\omega+\hat{t}\alpha e^{-\hat{\vp}}\rho+2n\alpha\ddbar\hat{\vp}\right)\wedge\dbar\omega^{n-2} \\
     & -\dbar v\wedge\left(e^{\hat{\vp}}\omega+\hat{t}\alpha e^{-\hat{\vp}}\rho+2n\alpha\ddbar\hat{\vp}\right)\wedge\sqrt{-1}\de\omega^{n-2}.
\end{split}
\end{equation}
Thus  $L^{*}$ has no zero order terms.   By the strong maximum principle, we see that
\begin{equation}\label{Kernel of L*}
\textrm{Ker}L^{*} \subset \{ \text{constants} \}.
\end{equation}
As a consequence,
$\textrm{Ker}L \subset  \{ cu_{0} ~|~ c\in\mathbf{R} \}$ for some smooth function $u_0$  by $\textrm{Ind}(L)=0$.
On the other hand,  again by the strong maximum principle, we may assume that $u_{0}>0$. Thus
\begin{equation*}
u_{0} \notin \{ u\in C^{2,\beta}(M) ~|~ \int_{M}ue^{-\hat{\vp}}\omega^{n}=0 \},
\end{equation*}
which  implies  $\textrm{Ker}L=0$, and so $L$ is injective.

Next, for any
\begin{equation*}
w \in \{ v\in C^{\beta}(M) ~|~ \int_{M}v\omega^{n}=0 \},
\end{equation*}
by the Fredholm alternative and regularity theory of elliptic equations, there exists a function $\ti{u}\in C^{2,\beta}(M)$ such that $L\ti{u}=w$. It then follows that
\begin{equation*}
L(\ti{u}+c_{0}u_{0}) = w \text{~and~} \ti{u}+c_{0}u_{0} \in \{ u\in C^{2,\beta}(M) ~|~ \int_{M}ue^{-\hat{\vp}}\omega^{n}=0 \},
\end{equation*}
where
\begin{equation*}
c_{0} = -\frac{\int_{M}\ti{u}e^{-\hat{\vp}}\omega^{n}}{\int_{M}u_{0}e^{-\hat{\vp}}\omega^{n}}.
\end{equation*}
This implies that $L$ is surjective.

\textbf{Closeness.} First we prove the $C^{0}$-estimate along (\ref{Fu-Yau equation t}). Recalling $\vp_{0}=h+\log\|e^{-h}\|_{L^{1}}-\log A$ and $A\ll1$, we have $\sup_{M}e^{-\vp_{0}}\leq M_{0}A$. We claim
\begin{claim}
\begin{equation*}
\sup_{M}e^{-\vp_{t}}\leq 2M_{0}A, ~\forall~ t\in[0,1].
\end{equation*}
\end{claim}
If the claim is false, there exists $\ti{t}\in(0,1)$ such that
\begin{equation}\label{Claim 2 equation 1}
\sup_{M}e^{-\vp_{\ti{t}}} = 2M_{0}A.
\end{equation}
Since $A\ll1$, we assume that $2M_{0}A\leq\delta_{0}$, where $\delta_{0}$ is the constant in Proposition \ref{Zero order estimate}. Then,  applying Proposition \ref{Zero order estimate} while $\rho$ and $\mu$ are replaced by
\begin{equation*}
t\rho \text{~and~} \frac{n\alpha(t-1)n!\ddbar h\wedge\ddbar h\wedge\omega^{n-2}}{\omega^{n}}+t\mu,
\end{equation*}
we obtain $\sup_{M}e^{-\vp_{\ti{t}}}\leq M_{0}A$, which contradicts with (\ref{Claim 2 equation 1}). Thus the claim is true.

By Claim 2, we see that Proposition \ref{Zero order estimate} and Proposition \ref{First and second order estimate} hold for $\vp_{t}$. As a consequence, we get the $C^{2}$-estimate for $\varphi_t$  along (\ref{Fu-Yau equation t}). Then combining the $C^{2,\alpha}$-estimate (cf. \cite[Theorem 1.1]{TWWY15}) and the bootstrapping argument, we complete the proof of closeness (for more details, we refer the reader to \cite[Section 5.2]{CHZ18}).

\textbf{Uniqueness.}  The uniqueness of  solutions of   (\ref{Fu-Yau equation})  can be proved by a similar argument of \cite[Section 5.3]{CHZ18} (also see  the proof of Theorem \ref{Uniqueness theorem alpha negative}  in  Section 7 below).     It suffices to prove that $\varphi_0$  of  (\ref{Fu-Yau equation t})   is unique  when $t=0$ by the estimates in  Proposition \ref{Zero order estimate} and   Proposition \ref{First and second order estimate}.  But the latter is guarantted by Remark \ref{unique-2-h}
in Section 7.
\end{proof}

\section{Zero order estimate (II)}
In this section, we improve  Proposition \ref{Zero order estimate} in the case of $\alpha<0$.   The key point is to drop the condition $e^{-\vp}\leq\delta_{0}$. We begin with the following lemma.

\begin{lemma}\label{Effective zero order estimate lemma}
Let $\alpha<0$ and $\vp$ be a smooth solution of (\ref{Fu-Yau equation}) satisfying (\ref{Elliptic condition}) and (\ref{L^n normalization condition}). There exists constant $A_{0}$ depending only on $\alpha$, $\rho$, $\mu$ and $(M,\omega)$ such that if any $A\leq A_{0}$, then
\begin{equation}\label{c0-2}
e^{-\vp} \leq C_{0}.
\end{equation}
\end{lemma}

\begin{proof}
The elliptic condition $\ti{\omega}\in\Gamma_{2}(M)$ implies that
\begin{equation*}
\int_{M}e^{-k\vp}\ti{\omega}\wedge\omega^{n-1} \geq 0.
\end{equation*}
Namely,
$$0 \leq  \int_{M}e^{-k\vp}(e^{\vp}\omega+\alpha e^{-\vp}\rho+2n\alpha\ddbar\vp)\wedge\omega^{n-1} .$$
By the Stokes' formula, for $k>1$, it follows that
\begin{equation}\label{Zero order estimate equation 2}
\begin{split}
0 \leq   {} & \int_{M}e^{-(k-1)\vp}\omega^{n}-|\alpha|\int_{M}e^{-(k+1)\vp}\rho\wedge\omega^{n-1} \\
          & +2n\alpha\int_{M}e^{-k\vp}\dbar\vp\wedge\sqrt{-1}\de\omega^{n-1} \\
          & -2n|\alpha|k\int_{M}e^{-k\vp}\sqrt{-1}\de\vp\wedge\db\vp\wedge\omega^{n-1} \\
  \leq {} & C\int_{M}\left(e^{-(k-1)\vp}+e^{-(k+1)\vp}\right)\omega^{n} \\
          & -\frac{2n\alpha}{k}\int_{M}e^{-k\vp}\ddbar\omega^{n-1} \\
          & -2n|\alpha|k\int_{M}e^{-k\vp}\sqrt{-1}\de\vp\wedge\db\vp\wedge\omega^{n-1}.
\end{split}
\end{equation}
Thus by the Sobolev inequality, we obtain
\begin{equation}\label{Zero order estimate equation 1}
\begin{split}
~    & \left(\int_{M}e^{-k\beta\vp}\omega^{n}\right)^{\frac{1}{\beta}}\\
\leq & C_{1}k\int_{M}e^{-(k-1)\vp}\omega^{n}+C_{1}k\int_{M}e^{-(k+1)\vp}\omega^{n}+C\int_{M}e^{-k\vp}\omega^{n},
\end{split}
\end{equation}
where $\beta=\frac{n}{n-1}$.

Next, we  prove that  $\|e^{-\vp}\|_{L^{n+1}}\leq C$  by  (\ref{Zero order estimate equation 1}).  In fact,   by taking $k=n$ in
(\ref{Zero order estimate equation 1})  and the H\"{o}lder inequality,    we see that
\begin{equation*}
\begin{split}
\|e^{-\vp}\|_{L^{n\beta}}^{n}
\leq {} & C_{1}n\|e^{-\vp}\|_{L^{n}}^{n-1}+C_{1}n\|e^{-\vp}\|_{L^{n}}\|e^{-\vp}\|_{L^{n\beta}}^{n}+C\|e^{-\vp}\|_{L^{n}} \\
\leq {} & CA^{n-1}+C_{1}nA\|e^{-\vp}\|_{L^{n\beta}}^{n}+CA.
\end{split}
\end{equation*}
Note that  $C_{1}nA\leq\frac{1}{2}$.   Then
\begin{equation*}
\|e^{-\vp}\|_{L^{n\beta}} \leq C.
\end{equation*}
Thus we get
\begin{equation}\label{ln-integral}
\|e^{-\vp}\|_{L^{n+1}} \leq \|e^{-\vp}\|_{L^{n\beta}} \leq C.
\end{equation}

Finally,
we use the iteration to obtain (\ref{c0-2}).
Let $H=e^{-\vp}+1$.  It suffices to prove that $\|H\|_{L^{\infty}}\leq C$.   By  (\ref{Zero order estimate equation 1}),  it is easy to see  that
\begin{equation*}
\|H\|_{L^{k\beta}} \leq (Ck)^{\frac{1}{k}}\|H\|_{L^{k+1}}^{\frac{k+1}{k}}.
\end{equation*}
For $j\geq0$, we define
\begin{equation*}
p_{j} = \beta^{j}+\frac{\beta}{\beta-1}, ~a_{j} = (Cp_{j}-p_{j})^{\frac{1}{p_{j}-1}} \text{~and~} b_{j} = \frac{p_{j}}{p_{j}-1}.
\end{equation*}
It then follows that
\begin{equation*}
\|H\|_{L^{p_{j+1}}} \leq a_{j}\|H\|_{L^{p_{j}}}^{b_{j}},
\end{equation*}
which implies
\begin{equation}\label{iteration-h}
\|H\|_{L^{p_{j+1}}} \leq a_{j}a_{j-1}^{b_{j}}\cdots a_{0}^{b_{j}\cdots b_{1}}\|H\|_{L^{p_{0}}}^{b_{j}\cdots b_{0}}.
\end{equation}
Note that  $\prod_{i=1}^{\infty}a_{i}<\infty$ and $\prod_{i=1}^{\infty}b_{i}<\infty$.  Thus  as $j\rightarrow\infty$,  we obtain
from (\ref{iteration-h}),
 \begin{equation*}
\|H\|_{L^{\infty}} \leq C\|H\|_{L^{n+1}}^{C}\le C'.
\end{equation*}
Here we used (\ref{ln-integral}).
\end{proof}

Now,  we  apply  Lemma \ref{Effective zero order estimate lemma}  and Lemma \ref{Zero order estimate lemma} to   improve  Proposition \ref{Zero order estimate} as follows.

\begin{proposition}\label{Stronger zero order estimate}
Let $\alpha<0$ and $\vp$ be a smooth solution of (\ref{Fu-Yau equation}) satisfying (\ref{Elliptic condition}) and (\ref{L^n normalization condition}). There exist constants $A_{0}$ and $M_{0}$ depending only on $\alpha$, $\rho$, $\mu$ and $(M,\omega)$ such that if $A\leq A_{0}$, then
\begin{equation*}
\frac{A}{M_{0}} \leq e^{-\vp} \leq M_{0}A.
\end{equation*}
\end{proposition}

\begin{proof}
By taking $f(\vp)=-e^{-k\vp}$  ($k\geq2$) in Lemma \ref{Zero order estimate lemma},   we have
\begin{equation*}
\begin{split}
& k\int_{M}e^{-k\vp}(e^{\vp}\omega+\alpha e^{-\vp}\rho)\wedge\sqrt{-1}\de\vp\wedge\dbar\vp\wedge\omega^{n-2} \\
\leq {} & -2k\int_{M}e^{-k\vp}\sqrt{-1}\de\vp\wedge(e^{\vp}\dbar\omega-\alpha e^{-\vp}\dbar\rho)\wedge\omega^{n-2} \\
        & -2k\int_{M}e^{-k\vp}\dbar\vp\wedge(e^{\vp}\omega-\alpha e^{-\vp}\rho)\wedge\sqrt{-1}\de\omega^{n-2}
          -2\int_{M}e^{-k\vp}\mu\frac{\omega^{n}}{n!}.
\end{split}
\end{equation*}
It follows that
\begin{equation*}
\int_{M}e^{-(k-1)\vp}|\de\vp|_{g}^{2}\omega^{n}
\leq  C\int_{M}\left(e^{-(k+1)\vp}|\de\vp|_{g}+e^{-k\vp}\right)\omega^{n}.
\end{equation*}
Combining this with Lemma \ref{Effective zero order estimate lemma} and the Cauchy-Shwarz inequality,  we get
\begin{equation*}
\begin{split}
\int_{M}e^{-(k-1)\vp}|\de\vp|_{g}^{2}\omega^{n}
\leq {} & C\int_{M}\left(e^{-(k+3)\vp}+e^{-k\vp}\right)\omega^{n} \\
\leq {} & C(C_{0}^{4}+C_{0})\int_{M}e^{-(k-1)\vp}\omega^{n}.
\end{split}
\end{equation*}
Thus by the Moser iteration,  we  derive
\begin{equation*}
\|e^{-\vp}\|_{L^{\infty}} \leq C\|e^{-\vp}\|_{L^{1}} \leq CA.
\end{equation*}
Note that  $A\ll 1$.   Hence $e^{-\vp}\ll 1$.    Now we can apply Proposition \ref{Zero order estimate} to obtain
\begin{equation*}
\frac{A}{M_{0}} \leq e^{-\vp} \leq M_{0}A.
\end{equation*}
\end{proof}

\section{First and second order estimates (II)}

In this section, we provide another proof to derive a prior $C^1,C^2$ estimates for $\varphi$ of (\ref{Fu-Yau equation}) in case of   $\alpha<0$, but without the  restriction condition  (\ref{Restrictions}).   For convenience, we say a constant $C$ is uniform if it depends only on $\alpha$, $\rho$, $\mu$ and $(M,\omega)$, and we use $C_{A}$ to denote a uniform constant depending on $A$.  The main goal in this section is to prove the following proposition.

\begin{proposition}\label{Effective second order estimate}
Let $\alpha<0$ and $\vp$ be a smooth solution of (\ref{Fu-Yau equation}) on a Hermitian manifold $(M,\omega)$, which satisfies (\ref{Elliptic condition}) and (\ref{L^n normalization condition}). Then
\begin{equation*}
\sup_{M}|\de\dbar\vp|_{g} \leq C_{A}\sup_{M}|\de\vp|_{g}^{2}+C_{A},
\end{equation*}
where $C_{A}$ is a uniform constant depending on $A$.
\end{proposition}

For simplicity, we write   (\ref{Fu-Yau equation t})  as
\begin{equation}\label{Fu-Yau equation 2-nd Hessian type}
\sigma_{2}(\ti{\omega}) = F,
\end{equation}
where
$$F= \frac{n(n-1)}{2}\left(e^{2\vp}+4|\alpha|e^\varphi|\de\vp|_{g}^{2}\right)+\frac{n(n-1)}{2}f$$
and
$f$ is defined by (\ref{Definition of f}).
Let $\lambda_{1}\geq\lambda_{2}\geq\cdots\geq\lambda_{n}$ be the eigenvalues of $\ti{\omega}$ with respect to
$\omega$. Since $\ti{\omega}\in\Gamma_{2}(M)$, by Proposition \ref{Stronger zero order estimate},  it is clear that
\begin{equation}\label{ddbar and lambda}
|\de\dbar\vp|_{g} \leq C_{A}\lambda_{1}+C_{A}.
\end{equation}

In Hermitian case, more troublesome terms will  appear when we commute the covariant derivatives (cf. (\ref{Commutation formulas}) below). To deal with these bad terms, we consider the following auxiliary function as in \cite{CHZ17},
\begin{equation*}
Q = \log\lambda_{1}+h(|\de\vp|_{g}^{2})+e^{B\vp},
\end{equation*}
where  $B$ is a constant to be determined later,
\begin{equation*}
h(t) = -\frac{1}{2}\log(2K-t) \text{~and~} K = \sup_{M}|\de\vp|_{g}^{2}+1.
\end{equation*}
By directly calculation, we have
\begin{equation}\label{h'}
\frac{1}{4K}\leq h'\leq \frac{1}{2K} \text{~and~} h''=2(h')^{2}.
\end{equation}

Let $x_{0}$ be the maximum point of $Q$. Around $x_{0}$, we choose holomorphic coordinate $(z^{1},z^{2},\cdots,z^{n})$ such that at $x_{0}$,
\begin{equation*}
g_{i\ov{j}} = \delta_{ij} \text{~and~}
\ti{g}_{i\ov{j}} = \delta_{ij}\lambda_{i}.
\end{equation*}
To prove Proposition \ref{Effective second order estimate}, by (\ref{ddbar and lambda}), it suffices to prove $\lambda_{1}\leq C_{A}K$. Without loss of generality, we assume that $\lambda_{1}\gg C_{A}K$.  Moreover, We may suppose  that $Q$ is smooth at $x_{0}$. Otherwise, we just need to apply a perturbation argument (cf. \cite{Sze15,STW17,CTW16}).

In the following calculation, we use the covariant derivatives with respect to the Chern connection $(\nabla, T^{\mathbb C}M)$ induced by $\omega$.  Let us recall the commutation formulas for covariant derivatives:
\begin{equation}\label{Commutation formulas}
\begin{split}
\vp_{i\ov{j}k} {} & = \vp_{ki\ov{j}}-T_{ki}^{p}\vp_{p\ov{j}}-R_{i\ov{j}k}^{\ \ \ \ p}\vp_{p}, \\
\vp_{i\ov{j}k\ov{l}} = \vp_{k\ov{l}i\ov{j}}-T_{ki}^{p}\vp_{p\ov{l}\ov{j}} {} & -\ov{T_{lj}^{q}}\vp_{k\ov{q}i}
                        +\vp_{p\ov{j}}R_{k\ov{l}i}^{\ \ \ \ p}-\vp_{p\ov{l}}R_{i\ov{j}k}^{\ \ \ \ p}
                       -T_{ik}^{p}\ov{T_{lj}^{q}}\vp_{p\ov{q}},
\end{split}
\end{equation}
where  $T_{ij}^{k}$ and $R_{i\ov{j}k}^{\ \ \ \ l}$ are components of torsion tensor and curvature tensor  induced by the Chern connection.
Let
\begin{equation*}
G^{i\ov{j}} = \frac{\de\sigma_{2}^{\frac{1}{2}}(\ti{\omega})}{\de\ti{g}_{i\ov{j}}},~
G^{i\ov{j},k\ov{l}} = \frac{\de^{2}\sigma_{2}^{\frac{1}{2}}(\ti{\omega})}{\de\ti{g}_{i\ov{j}}\de\ti{g}_{k\ov{l}}}.
\end{equation*}
The following lemmas are devoted to deriving lower bounds of $G^{k\ov{k}}\vp_{k\ov{k}}$, $G^{k\ov{k}}(|\de\vp|_{g}^{2})_{k\ov{k}}$ and $G^{k\ov{k}}(\lambda_{1})_{k\ov{k}}$.

\begin{lemma}\label{Lemma 1}
At $x_{0}$, we have
\begin{equation*}
G^{k\ov{k}}\vp_{k\ov{k}} \geq \sum_{k}G^{k\ov{k}}-CF^{\frac{1}{2}}
\end{equation*}
and
\begin{equation*}
\begin{split}
G^{k\ov{k}}(|\de\vp|_{g}^{2})_{k\ov{k}}
\geq {} & \frac{1}{2}\sum_{i}G^{k\ov{k}}\left(|\vp_{ik}|^{2}+|\vp_{i\ov{k}}|^{2}\right)
          -C_{A}K^{\frac{3}{2}} \\
        & -C_{A}K^{\frac{1}{2}}\sum_{i,k}\left(|\vp_{ik}|+|\vp_{i\ov{k}}|\right)-C\sum_{k}G^{k\ov{k}}.
\end{split}
\end{equation*}
\end{lemma}

\begin{proof}
 By the definition of $\ti{\omega}$, we have
\begin{equation*}
\begin{split}
2n\alpha G^{k\ov{k}}\vp_{k\ov{k}}
= {} & G^{k\ov{k}}\left(\ti{g}_{k\ov{k}}-e^{\vp}g_{k\ov{k}}-\alpha e^{-\vp}\rho_{k\ov{k}}\right) \\[2mm]
  =  {} & \sigma_{2}^{\frac{1}{2}}(\ti{\omega})-e^{\vp}\sum_{k}G^{k\ov{k}}-\alpha e^{-\vp}G^{k\ov{k}}\rho_{k\ov{k}} \\
  =  {} & F^{\frac{1}{2}}-e^{\vp}\sum_{k}G^{k\ov{k}}-\alpha e^{-\vp}G^{k\ov{k}}\rho_{k\ov{k}}.
\end{split}
\end{equation*}
Note that $A\ll1$ and $\alpha<0$.   Then by Proposition \ref{Stronger zero order estimate},  we get
\begin{equation*}
G^{k\ov{k}}\vp_{k\ov{k}} \geq \sum_{k}G^{k\ov{k}}-CF^{\frac{1}{2}}.
\end{equation*}

To prove the second inequality in the lemma, we compute
\begin{equation}\label{Lemma 1 equation 1}
\begin{split}
        & G^{k\ov{k}}(|\de\vp|_{g}^{2})_{k\ov{k}} \\
   = {} & \sum_{i}G^{k\ov{k}}\left(|\vp_{ik}|^{2}+|\vp_{i\ov{k}}|^{2}\right)
          +2\textrm{Re}\left(\sum_{i}G^{k\ov{k}}\vp_{i}\vp_{\ov{i}k\ov{k}}\right) \\
\geq {} & \sum_{i}G^{k\ov{k}}\left(|\vp_{ik}|^{2}+|\vp_{i\ov{k}}|^{2}\right)
          +2\textrm{Re}\left(\sum_{i}G^{k\ov{k}}\vp_{i}\vp_{k\ov{k}\ov{i}}\right) \\[1mm]
        & -C\sum_{i}G^{k\ov{k}}\left(|\vp_{ik}|+|\vp_{i\ov{k}}|\right) \\
\geq {} & \frac{1}{2}\sum_{i}G^{k\ov{k}}\left(|\vp_{ik}|^{2}+|\vp_{i\ov{k}}|^{2}\right)
          +2\textrm{Re}\left(\sum_{i}G^{k\ov{k}}\vp_{i}\vp_{k\ov{k}\ov{i}}\right)
          -C\sum_{k}G^{k\ov{k}}.
\end{split}
\end{equation}
On the other hand, by differentiating (\ref{Fu-Yau equation 2-nd Hessian type}) along $\nabla_{\ov{i}}$, we see that
\begin{equation*}
2\textrm{Re}\left(\sum_{i}G^{k\ov{k}}\vp_{i}\vp_{k\ov{k}\ov{i}}\right)
= 2\textrm{Re}\left(\sum_{i}\vp_{i}(F^{\frac{1}{2}})_{\ov{i}}\right)
= \frac{1}{F^{\frac{1}{2}}}\textrm{Re}\left(\sum_{i}\vp_{i}F_{\ov{i}}\right).
\end{equation*}
 Since
\begin{equation*}
F \geq \frac{1}{C}(e^{2\vp}+e^{\vp}|\de\vp|_{g}^{2})
\end{equation*}
and
\begin{equation*}
|F_{\ov{i}}| \leq Ce^{2\vp}|\de\vp|_{g}+Ce^{\vp}|\de\vp|_{g}^{3}
                  +Ce^{\vp}(|\de\vp|_{g}+1)\sum_{i,k}\left(|\vp_{ik}|+|\vp_{i\ov{k}}|\right)+C,
\end{equation*}
we get
\begin{equation*}
\begin{split}
2\textrm{Re}\left(\sum_{i}G^{k\ov{k}}\vp_{i}\vp_{\ov{i}k\ov{k}}\right)
\geq {} & -Ce^{\frac{3}{2}\vp}|\de\vp|_{g}-Ce^{\frac{1}{2}\vp}|\de\vp|_{g}^{3} \\
        & +Ce^{\frac{1}{2}\vp}(|\de\vp|_{g}+1)\sum_{i,k}\left(|\vp_{ik}|+|\vp_{i\ov{k}}|\right)-C.
\end{split}
\end{equation*}
Hence, substituting this into (\ref{Lemma 1 equation 1}), we obtain
\begin{equation*}
\begin{split}
G^{k\ov{k}}(|\de\vp|_{g}^{2})_{k\ov{k}}
\geq {} & \frac{1}{2}\sum_{i}G^{k\ov{k}}\left(|\vp_{ik}|^{2}+|\vp_{i\ov{k}}|^{2}\right)
          -Ce^{\frac{3}{2}\vp}|\de\vp|_{g}-Ce^{\frac{1}{2}\vp}|\de\vp|_{g}^{3} \\[1mm]
        & -Ce^{\frac{1}{2}\vp}(|\de\vp|_{g}+1)\sum_{i,k}\left(|\vp_{ik}|+|\vp_{i\ov{k}}|\right)-C\sum_{k}G^{k\ov{k}}-C \\
\geq {} & \frac{1}{2}\sum_{i}G^{k\ov{k}}\left(|\vp_{ik}|^{2}+|\vp_{i\ov{k}}|^{2}\right)
          -C_{A}K^{\frac{3}{2}} \\[1mm]
        & -C_{A}K^{\frac{1}{2}}\sum_{i,k}\left(|\vp_{ik}|+|\vp_{i\ov{k}}|\right)-C\sum_{k}G^{k\ov{k}}.
\end{split}
\end{equation*}
\end{proof}

\begin{lemma}\label{Lemma 2}
At $x_{0}$, we have
\begin{equation*}
\begin{split}
G^{k\ov{k}}(\lambda_{1})_{k\ov{k}}
\geq {} & -G^{i\ov{j},k\ov{l}}\nabla_{1}\ti{g}_{i\ov{j}}\nabla_{\ov{1}}\ti{g}_{p\ov{q}}
          -C\sum_{k}G^{k\ov{k}}|\vp_{1\ov{1}k}|-C_{A}\sum_{k}|\vp_{k1\ov{1}}| \\
        & -C_{A}(\lambda_{1}+K)\sum_{i}G^{i\ov{i}}-C_{A}K^{-\frac{1}{2}}\sum_{i,k}|\vp_{i\ov{k}}|^{2}
          -C_{A}K^{\frac{3}{2}}.
\end{split}
\end{equation*}
\end{lemma}

\begin{proof}
By the formulas for the derivatives of $\lambda_{1}$ (cf. \cite{Sze15,STW17,CTW16}), we have
\begin{equation}\label{Lemma 2 equation 1}
\begin{split}
         G^{k\ov{k}}(\lambda_{1})_{k\ov{k}}
   &=   G^{k\ov{k}}\nabla_{\ov{k}}\nabla_{k}\ti{g}_{1\ov{1}}
          +2\sum_{j>1}G^{k\ov{k}}\frac{|\nabla_{k}\ti{g}_{1\ov{j}}|^{2}}{\lambda_{1}-\lambda_{j}} \\
& \geq   2n\alpha G^{k\ov{k}}\vp_{1\ov{1}k\ov{k}}
          +\frac{2}{n\lambda_{1}}\sum_{j>1}G^{k\ov{k}}|\nabla_{k}\ti{g}_{i\ov{j}}|^{2}
          -Ce^{\vp}(\lambda_{1}+K)\sum_{i}G^{i\ov{i}} \\
&\geq   2n\alpha G^{k\ov{k}}\vp_{1\ov{1}k\ov{k}}
          +\frac{2}{n\lambda_{1}}\sum_{j>1}G^{k\ov{k}}|\vp_{j\ov{k}\ov{1}}|^{2}
          -Ce^{\vp}(\lambda_{1}+K)\sum_{i}G^{i\ov{i}} \\
\end{split}
\end{equation}
On the other hand, by differentiating (\ref{Fu-Yau equation 2-nd Hessian type}) along $\nabla_{\ov{1}}\nabla_{1}$, we  have
\begin{equation*}
G^{k\ov{k}}\nabla_{\ov{1}}\nabla_{1}(e^{\vp}g_{k\ov{k}}+\alpha e^{-\vp}\rho_{k\ov{k}}+2n\alpha\vp_{k\ov{k}})
= (F^{\frac{1}{2}})_{1\ov{1}}-G^{i\ov{j},k\ov{l}}\nabla_{1}\ti{g}_{i\ov{j}}\nabla_{\ov{1}}\ti{g}_{p\ov{q}}.
\end{equation*}
Then  by  the commutation formula (\ref{Commutation formulas}), we see that
\begin{equation*}
\begin{split}
2n\alpha G^{k\ov{k}}\vp_{1\ov{1}k\ov{k}}
\geq {} & -C\sum_{i,k}G^{k\ov{k}}|\vp_{i\ov{k}\ov{1}}|-Ce^{\vp}(\lambda_{1}+K)\sum_{i}G^{i\ov{i}} \\
        & +(F^{\frac{1}{2}})_{1\ov{1}}-G^{i\ov{j},k\ov{l}}\nabla_{1}\ti{g}_{i\ov{j}}\nabla_{\ov{1}}\ti{g}_{p\ov{q}}.
\end{split}
\end{equation*}
Hence, substituting this into (\ref{Lemma 2 equation 1}) and  using the Cauchy-Schwarz inequality, we obtain
\begin{equation}\label{Lemma 2 equation 4}
\begin{split}
G^{k\ov{k}}(\lambda_{1})_{k\ov{k}}
\geq {} & -C\sum_{k}G^{k\ov{k}}|\vp_{1\ov{1}k}|-Ce^{\vp}(\lambda_{1}+K)\sum_{i}G^{i\ov{i}} \\
        & +(F^{\frac{1}{2}})_{1\ov{1}}-G^{i\ov{j},k\ov{l}}\nabla_{1}\ti{g}_{i\ov{j}}\nabla_{\ov{1}}\ti{g}_{p\ov{q}}.
\end{split}
\end{equation}

Next, we  deal with the term $(F^{\frac{1}{2}})_{1\ov{1}}$.     Clearly,
\begin{equation}\label{Lemma 2 equation 2}
(F^{\frac{1}{2}})_{1\ov{1}} = \frac{1}{2F^{\frac{1}{2}}}\left(F_{1\ov{1}}-\frac{|F_{1}|^{2}}{2F}\right).
\end{equation}
For the first term of (\ref{Lemma 2 equation 2}), we compute
\begin{equation*}
\begin{split}
F_{1\ov{1}} = {} & 2n(n-1)|\alpha|(e^{\vp}|\de\vp|_{g}^{2})_{1\ov{1}}+\frac{n(n-1)}{2}(e^{2\vp})_{1\ov{1}}+\frac{n(n-1)}{2}f_{1\ov{1}} \\[1mm]
\geq {} & 2n(n-1)|\alpha|e^{\vp}\sum_{k}\left(|\vp_{k1}|^{2}+|\vp_{k\ov{1}}|^{2}\right)-Ce^{\vp}|\de\vp|_{g}\sum_{k}|\vp_{k1\ov{1}}| \\
        & -Ce^{\vp}|\de\vp|_{g}^{2}|\vp_{1\ov{1}}|-Ce^{\vp}|\de\vp|_{g}^{4}
          -Ce^{\vp}|\de\vp|_{g}^{2}\sum_{k}\left(|\vp_{k1}|+|\vp_{k\ov{1}}|\right) \\
        & -Ce^{2\vp}|\vp_{1\ov{1}}|+\frac{n(n-1)}{2}f_{1\ov{1}} \\
\geq {} & \frac{3}{2}n(n-1)|\alpha|e^{\vp}\sum_{k}\left(|\vp_{k1}|^{2}+|\vp_{k\ov{1}}|^{2}\right)-Ce^{\vp}(|\de\vp|_{g}+1)\sum_{k}|\vp_{k1\ov{1}}| \\
        & -Ce^{\vp}|\de\vp|_{g}^{4}-Ce^{3\vp}.
\end{split}
\end{equation*}
For the second term of (\ref{Lemma 2 equation 2}), we have
\begin{equation*}
\begin{split}
F_{1} = {} & 2n(n-1)|\alpha|e^{\vp}\nabla_{1}|\de\vp|_{g}^{2}+2n(n-1)|\alpha|e^{\vp}|\de\vp|_{g}^{2}\vp_{1} \\
           & +n(n-1)e^{2\vp}\vp_{1}+\frac{n(n-1)}{2}f_{1},
\end{split}
\end{equation*}
which implies
\begin{equation*}
\begin{split}
|F_{1}|^{2} \leq {} & \left(4+\frac{1}{200}\right)n^{2}(n-1)^{2}\alpha^{2}e^{2\vp}|\de\vp|_{g}^{2}\sum_{k}|\vp_{k1}|^{2} \\
                    & +Ce^{2\vp}|\de\vp|_{g}^{2}\sum_{k}|\vp_{k\ov{1}}|^{2}
                      +Ce^{2\vp}|\de\vp|_{g}^{6}+Ce^{4\vp}|\de\vp|_{g}^{2}+C|f_{1}|^{2} \\
            \leq {} & \left(4+\frac{1}{100}\right)n^{2}(n-1)^{2}\alpha^{2}e^{2\vp}|\de\vp|_{g}^{2}\sum_{k}|\vp_{k1}|^{2}
                      +Ce^{2\vp}\sum_{k}|\vp_{k1}|^{2} \\
                    & +Ce^{2\vp}(|\de\vp|_{g}^{2}+1)\sum_{k}|\vp_{k\ov{1}}|^{2}
                      +Ce^{2\vp}|\de\vp|_{g}^{6}+Ce^{4\vp}|\de\vp|_{g}^{2}+C. \\
\end{split}
\end{equation*}
On the other hand,  by the definition of $F$, we have
\begin{equation*}
2F \geq \frac{99}{100}n(n-1)\left(e^{2\vp}+4|\alpha|e^{\vp}|\de\vp|_{g}^{2}\right).
\end{equation*}
Thus, substituting these estimates into (\ref{Lemma 2 equation 2}),  we get
\begin{equation}\label{Lemma 2 equation 3}
\begin{split}
& (F^{\frac{1}{2}})_{1\ov{1}} \\
\geq {} & -\frac{C}{F^{\frac{1}{2}}}\left(e^{\vp}(|\de\vp|_{g}+1)\sum_{k}|\vp_{k1\ov{1}}|+e^{\vp}\sum_{k,l}|\vp_{k\ov{l}}|^{2}
          +e^{\vp}|\de\vp|_{g}^{4}+e^{3\vp}\right) \\
\geq {} & -Ce^{\frac{1}{2}\vp}\sum_{k}|\vp_{k1\ov{1}}|-Ce^{\vp}F^{-\frac{1}{2}}\sum_{k,l}|\vp_{k\ov{l}}|^{2}
          -Ce^{\frac{1}{2}\vp}|\de\vp|_{g}^{3}-Ce^{2\vp}.
\end{split}
\end{equation}
 We note that a  similar estimate  of  (\ref{Lemma 2 equation 3}) was also appeared    in \cite{PPZ16b}.

Combining (\ref{Lemma 2 equation 4}) and (\ref{Lemma 2 equation 3}), we finally prove that
\begin{equation*}
\begin{split}
        & \ G^{k\ov{k}}(\lambda_{1})_{k\ov{k}} \\[1mm]
\geq {} & -G^{i\ov{j},k\ov{l}}\nabla_{1}\ti{g}_{i\ov{j}}\nabla_{\ov{1}}\ti{g}_{p\ov{q}}
          -C\sum_{k}G^{k\ov{k}}|\vp_{1\ov{1}k}|-Ce^{\frac{1}{2}\vp}\sum_{k}|\vp_{k1\ov{1}}| \\
        & -Ce^{\vp}(\lambda_{1}+K)\sum_{i}G^{i\ov{i}}-Ce^{\vp}F^{-\frac{1}{2}}\sum_{i,k}|\vp_{i\ov{k}}|^{2}
          -Ce^{\frac{1}{2}\vp}|\de\vp|_{g}^{3}-Ce^{2\vp} \\
\geq {} & -G^{i\ov{j},k\ov{l}}\nabla_{1}\ti{g}_{i\ov{j}}\nabla_{\ov{1}}\ti{g}_{p\ov{q}}
          -C\sum_{k}G^{k\ov{k}}|\vp_{1\ov{1}k}|-C_{A}\sum_{k}|\vp_{k1\ov{1}}| \\
        & -C_{A}(\lambda_{1}+K)\sum_{i}G^{i\ov{i}}-C_{A}K^{-\frac{1}{2}}\sum_{i,k}|\vp_{i\ov{k}}|^{2}
          -C_{A}K^{\frac{3}{2}}.
\end{split}
\end{equation*}
\end{proof}

Using the above lemmas, we prove the following lower bound of $G^{k\ov{k}}Q_{k\ov{k}}$ at $x_{0}$.
\begin{lemma}
At $x_{0}$, for any $\ve\in(0,1)$, we have
\begin{equation}\label{Lower bound}
\begin{split}
0 \geq G^{k\ov{k}}Q_{k\ov{k}}
  \geq {} & -\frac{G^{i\ov{j},k\ov{l}}\nabla_{1}\ti{g}_{i\ov{j}}\nabla_{\ov{1}}\ti{g}_{p\ov{q}}}{\lambda_{1}}
            -(1+\ve)\frac{G^{k\ov{k}}|\nabla_{k}\ti{g}_{1\ov{1}}|^{2}}{\lambda_{1}^{2}} \\
          & +\frac{h'}{4}\sum_{i}G^{k\ov{k}}\left(|\vp_{ik}|^{2}+|\vp_{i\ov{k}}|^{2}\right)
            +h''G^{k\ov{k}}|\nabla_{k}|\de\vp|_{g}^{2}|^{2} \\
          & +B^{2}e^{B\vp}G^{k\ov{k}}|\vp_{k}|^{2}+\left(\frac{B}{2}e^{B\vp}-\frac{C_{A}}{\ve}\right)\sum_{k}G^{k\ov{k}}.
\end{split}
\end{equation}
\end{lemma}

\begin{proof}
By \eqref{h'},  Lemma \ref{Lemma 1} and Lemma \ref{Lemma 2}, we see that
\begin{equation}\label{Lower bound equation 1}
\begin{split}
& G^{k\ov{k}}Q_{k\ov{k}}(x_{0})\\
 \geq {} & -\frac{G^{i\ov{j},k\ov{l}}\nabla_{1}\ti{g}_{i\ov{j}}\nabla_{\ov{1}}\ti{g}_{p\ov{q}}}{\lambda_{1}}
             -\frac{G^{k\ov{k}}|\nabla_{k}\ti{g}_{1\ov{1}}|^{2}}{\lambda_{1}^{2}}
             -\frac{C\sum_{k}G^{k\ov{k}}|\vp_{1\ov{1}k}|}{\lambda_{1}} \\
           & -\frac{C_{A}\sum_{k}|\vp_{k1\ov{1}}|}{\lambda_{1}}
             +\frac{h'}{2}\sum_{i}G^{k\ov{k}}\left(|\vp_{ik}|^{2}+|\vp_{i\ov{k}}|^{2}\right)
              -C_{A}K^{\frac{3}{2}}h'\\
           & +h''G^{k\ov{k}}|\nabla_{k}|\de\vp|_{g}^{2}|^{2}
             -C_{A}K^{\frac{1}{2}}h'\sum_{i,k}\left(|\vp_{ik}|+|\vp_{i\ov{k}}|\right)
             -C_{A}K^{\frac{1}{2}} \\
           & +B^{2}e^{B\vp}G^{k\ov{k}}|\vp_{k}|^{2}
             +\left(Be^{B\vp}-C_{A}\right)\sum_{k}G^{k\ov{k}}-CBe^{B\vp}F^{\frac{1}{2}}.
\end{split}
\end{equation}
On the other hand, by the definition of $\ti{\omega}$, we have
\begin{equation}\label{the first order derivative of tildeomega}
2n\alpha\vp_{1\ov{1}k} = \nabla_{k}\ti{g}_{1\ov{1}}-e^{\vp}\vp_{k}-\alpha(e^{-\vp}\rho_{1\ov{1}})_{k},
\end{equation}
which implies
\begin{equation}\label{Lower bound equation 2}
\begin{split}
-\frac{C\sum_{k}G^{k\ov{k}}|\vp_{1\ov{1}k}|}{\lambda_{1}}
\geq {} & -\frac{C}{\lambda_{1}}\sum_{k}G^{k\ov{k}}(|\nabla_{k}\ti{g}_{1\ov{1}}|+e^{\vp}K) \\
\geq {} & -\ve\frac{G^{k\ov{k}}|\nabla_{k}\ti{g}_{1\ov{1}}|^{2}}{\lambda_{1}^{2}}
          -\frac{C_{A}}{\ve}\sum_{k}G^{k\ov{k}},
\end{split}
\end{equation}
where $\ve\in(0,1)$. Note that  $\nabla_{k}Q(x_{0})=0$. Then
\begin{equation}\label{dQ=0}
\frac{\nabla_{k}\ti{g}_{1\ov{1}}}{\lambda_{1}} = -h'\nabla_{k}|\de\vp|_{g}^{2}-Be^{B\vp}\vp_{k}.
\end{equation}
Thus combining this with (\ref{the first order derivative of tildeomega}), it follows that
\begin{equation*}
-\frac{C_{A}\sum_{k}|\vp_{1\ov{1}k}|}{\lambda_{1}}
\geq -C_{A}K^{\frac{1}{2}}h'\sum_{i,k}\left(|\vp_{ik}|+|\vp_{i\ov{k}}|\right)-C_{A}K^{\frac{1}{2}}Be^{B\vp}.
\end{equation*}
Hence, substituting the above estimates into (\ref{Lower bound equation 1}), we get
\begin{equation}\label{Lower bound equation 3}
\begin{split}
0  \geq {} & -\frac{G^{i\ov{j},k\ov{l}}\nabla_{1}\ti{g}_{i\ov{j}}\nabla_{\ov{1}}\ti{g}_{p\ov{q}}}{\lambda_{1}}
             -(1+\ve)\frac{G^{k\ov{k}}|\nabla_{k}\ti{g}_{1\ov{1}}|^{2}}{\lambda_{1}^{2}}
             +h''G^{k\ov{k}}|\nabla_{k}|\de\vp|_{g}^{2}|^{2}\\
           & +\frac{h'}{2}\sum_{i}G^{k\ov{k}}\left(|\vp_{ik}|^{2}+|\vp_{i\ov{k}}|^{2}\right)
              -C_{A}K^{\frac{1}{2}}h'\sum_{i,k}\left(|\vp_{ik}|+|\vp_{i\ov{k}}|\right) \\
           & +B^{2}e^{B\vp}G^{k\ov{k}}|\vp_{k}|^{2}+\left(Be^{B\vp}-\frac{C_{A}}{\ve}\right)\sum_{k}G^{k\ov{k}} \\
           & -C_{A}K^{\frac{3}{2}}h'-C_{A}K^{\frac{1}{2}}
             -CBe^{B\vp}F^{\frac{1}{2}}-C_{A}K^{\frac{1}{2}}Be^{B\vp}.
\end{split}
\end{equation}
By the definitions of $G^{k\ov{k}}$, it is clear that
\begin{equation*}
\sum_{k}G^{k\ov{k}} =\frac{n-1}{2F^{\frac{1}{2}}}\lambda_{1}+(n-1)G^{1\ol{1}}\geq \frac{n-1}{2F^{\frac{1}{2}}}\lambda_{1}.
\end{equation*}
Combining this with \eqref{h'} and $G^{k\ol{k}}\geq \frac{F^{\frac{1}{2}}}{n\lambda_{1}}$, it  follows that
\begin{equation*}
\begin{split}
         C_{A}K^{\frac{1}{2}}h'\sum_{i,k}\left(|\vp_{ik}|+|\vp_{i\ov{k}}|\right)
\leq {} & \frac{h'}{10}\frac{F^{\frac{1}{2}}}{n\lambda_{1}}\sum_{i,k}\left(|\vp_{ik}|^{2}+|\vp_{i\ov{k}}|^{2}\right)
          +C_{A}Kh'\frac{\lambda_{1}}{F^{\frac{1}{2}}} \\
\leq {} & \frac{h'}{10}\sum_{i}G^{k\ov{k}}\left(|\vp_{ik}|^{2}+|\vp_{i\ov{k}}|^{2}\right)+C_{A}\sum_{k}G^{k\ov{k}}.
\end{split}
\end{equation*}
By the definition of $F$ and $\lambda_{1}\gg C_{A}K$, we have
\begin{equation*}
C_{A}K^{\frac{3}{2}}h'+C_{A}K^{\frac{1}{2}}+CBe^{B\vp}F^{\frac{1}{2}}+C_{A}K^{\frac{1}{2}}Be^{B\vp}
\leq C_{A}K^{\frac{1}{2}}Be^{B\vp}
\end{equation*}
and
\begin{equation*}
C_{A}K^{\frac{1}{2}}Be^{B\vp}
\leq  \frac{n-1}{4F^{\frac{1}{2}}}\lambda_{1}Be^{B\vp}
\leq  \frac{B}{2}e^{B\vp}\sum_{k}G^{k\ov{k}}.
\end{equation*}
Substituting these estimates into (\ref{Lower bound equation 3}), we obtain (\ref{Lower bound}) immediately.
\end{proof}

\begin{proof}[Proof of Proposition \ref{Effective second order estimate}] By (\ref{dQ=0}) and \eqref{h'}, we have
\begin{equation}\label{GkkDQ=0}
\begin{split}
(1+\ve)\frac{G^{k\ov{k}}|\nabla_{k} \ti{g}_{1\ol{1}}|^{2}}{\lambda_{1}^{2}}
   = {} & (1+\ve)G^{k\ov{k}}|h'\nabla_{k}|\de\vp|^{2}_{g}+Be^{B\vp}\vp_{k}|^{2} \\
\leq {} & 2(h')^{2}G^{k\ol{k}}|\nabla_{k}|\de \vp|^{2}_{g}|^{2}+CB^{2}e^{2B\vp}G^{k\ov{k}}|\vp_{k}|^{2} \\
\leq {} & h''G^{k\ol{k}}|\nabla_{k}|\de\vp|_{g}^{2}|^{2}+CB^{2}e^{2B\vp}G^{k\ov{k}}|\vp_{k}|^{2}.
\end{split}
\end{equation}
Since $G^{i\ov{j},k\ov{l}}\nabla_{1}\ti{g}_{i\ov{j}}\nabla_{\ov{1}}\ti{g}_{p\ov{q}}\leq0$,  by (\ref{Lower bound}) (taking $\ve=\frac{1}{2}$), we get
\begin{equation}\label{Inequality for i>1}
\begin{split}
0 \geq {} & \frac{h'}{8}\sum_{i}G^{k\ov{k}}\left(|\vp_{ik}|^{2}+|\vp_{i\ov{k}}|^{2}\right)-C_{A}KB^{2}e^{2B\vp}\sum_{k}G^{k\ov{k}} \\
  \geq {} & \frac{1}{16K}\sum_{k}G^{k\ol{k}}|\vp_{k\ol{k}}|^{2}-C_{A}B^{2}e^{2B\vp}K\sum_{k}G^{k\ov{k}}.
\end{split}
\end{equation}
On other hand hand, for the $\sigma_2$ function, we have (cf. \cite[Theorem 1]{LT94})),
$$G^{i\ol{i}}\geq \frac{1}{C}\sum_{k}G^{k\ol{k}}, ~\forall ~i\geq 2.$$
Substituting this into (\ref{Inequality for i>1}), it follows that
\begin{equation*}
0 \geq \frac{1}{16K}G^{i\ov{i}}|\vp_{i\ov{i}}|^{2}-C_{A}B^{2}e^{2B\vp}G^{i\ov{i}}.
\end{equation*}
Hence, we obtain
\begin{equation}\label{bounded of lambdak}
\lambda_{i} \leq C_{A}|\vp_{i\ov{i}}|+C_{A} \leq C_{A,B}K,  ~\forall i\geq 2,
\end{equation}
where $C_{A,B}$ is a uniform constant depending on $A$ and $B$.

By \eqref{Lower bound} and \eqref{GkkDQ=0} for $k=1$, we have
\begin{equation}\label{Lower bound 1}
\begin{split}
0
  \geq {} & -\frac{G^{i\ov{j},k\ov{l}}\nabla_{1}\ti{g}_{i\ov{j}}\nabla_{\ov{1}}\ti{g}_{k\ov{l}}}{\lambda_{1}}
            -(1+\ve)\sum_{k\geq2}\frac{G^{k\ov{k}}|\nabla_{k}\ti{g}_{1\ov{1}}|^{2}}{\lambda_{1}^{2}} \\
          & +\frac{h'}{8}\sum_{i}G^{k\ov{k}}\left(|\vp_{ik}|^{2}+|\vp_{i\ov{k}}|^{2}\right)
            +\sum_{k\geq 2}h''G^{k\ov{k}}|\nabla_{k}|\de\vp|_{g}^{2}|^{2} \\
          & +B^{2}e^{B\vp}G^{k\ov{k}}|\vp_{k}|^{2}+\left(\frac{B}{2}e^{B\vp}-\frac{C_{A}}{\ve}\right)\sum_{k}G^{k\ov{k}}
            -CKB^2e^{2B\vp}G^{1\ol{1}}.
\end{split}
\end{equation}
We need to deal with bad third order term
\begin{equation}\label{Bad third order term}
\begin{split}
& (1+\ve)\sum_{k\geq2}\frac{G^{k\ov{k}}|\nabla_{k}\ti{g}_{1\ov{1}}|^{2}}{\lambda_{1}^{2}} \\
= {} & (1-2\ve)\sum_{k\geq2}\frac{G^{k\ov{k}}|\nabla_{k}\ti{g}_{1\ov{1}}|^{2}}{\lambda_{1}^{2}}
  +3\ve\sum_{k\geq2}\frac{G^{k\ov{k}}|\nabla_{k}\ti{g}_{1\ov{1}}|^{2}}{\lambda_{1}^{2}}
\end{split}
\end{equation}
For the first term of (\ref{Bad third order term}),  we use  \eqref{bounded of lambdak} to see that
\begin{equation}\label{third 1}
\begin{split}
        & (1-2\ve)\sum_{k\geq2}\frac{G^{k\ov{k}}|\nabla_{k}\ti{g}_{1\ov{1}}|^{2}}{\lambda_{1}^{2}} \\
\leq {} & (1-\ve)\sum_{k\geq2}\frac{G^{k\ov{k}}|\nabla_{1}\ti{g}_{k\ov{1}}|^{2}}{\lambda_{1}^{2}}+\frac{C}{\ve}\sum_{k}G^{k\ov{k}} \\
\leq {} & (1-\ve)\sum_{k\geq2}\frac{\lambda_{1}+C_{A,B}K}{\lambda_{1}^{2}}|\nabla_{1}\ti{g}_{k\ov{1}}|^{2}+\frac{C}{\ve}\sum_{k}G^{k\ov{k}} \\
\leq {} & -\sum_{k\geq2}\frac{G^{1\ov{k},k\ov{1}}|\nabla_{1}\ti{g}_{k\ov{1}}|^{2}}{\lambda_{1}}+\frac{C}{\ve}\sum_{k}G^{k\ov{k}},
\end{split}
\end{equation}
as long as $\lambda_{1}\geq\frac{C_{A,B}K}{\ve}$. For the second term of (\ref{Bad third order term}),  we use (\ref{dQ=0}) to get
\begin{equation}\label{third 2}
\begin{split}
        &  3\ve\sum_{k\geq2}\frac{G^{k\ov{k}}|\nabla_{k}\ti{g}_{1\ov{1}}|^{2}}{\lambda_{1}^{2}}\\
\leq {} &  6\ve(h')^{2}\sum_{k\geq 2}G^{k\ov{k}}|\nabla_{k}|\de\vp|_{g}^{2}|^{2}+6\ve B^{2}e^{2B\vp}\sum_{k\geq 2}G^{k\ol{k}}|\vp_{k}|^{2}.
\end{split}
\end{equation}
Thus substituting  \eqref{third 1},  \eqref{third 2} and \eqref{Bad third order term} into  \eqref{Lower bound 1},
we obtain
\begin{equation*}
\begin{split}
0 \geq {} & \frac{h'}{8}\sum_{i}G^{k\ov{k}}\left(|\vp_{ik}|^{2}+|\vp_{i\ov{k}}|^{2}\right)
            +\left(B^{2}e^{B\vp}-6\ve B^{2}e^{2B\vp}\right)G^{k\ov{k}}|\vp_{k}|^{2} \\
          & +\left(\frac{B}{2}e^{B\vp}-\frac{C_{A}}{\ve}\right)\sum_{k}G^{k\ov{k}}-C_{A}KB^2e^{2B\vp}G^{1\ol{1}}.
\end{split}
\end{equation*}
Choose $B=12C_{A}+1$ and $\ve=\frac{e^{-B\vp({x_{0}})}}{6}$, so that
\begin{equation*}
B^{2}e^{B\vp}-6\ve B^{2}e^{2B\vp}=0 \text{~and~} \left(\frac{B}{2}e^{B\vp}-\frac{C_{A}}{\ve}\right)\sum_{k}G^{k\ov{k}} \geq 0.
\end{equation*}
Hence,
\begin{equation*}
\frac{h'}{8}G^{1\ov{1}}|\vp_{1\ov{1}}|^{2}\leq C_{A}KB^2e^{2B\vp}G^{1\ol{1}},
\end{equation*}
which implies $\lambda_{1}\leq C_{A}K$. We complete the proof.
\end{proof}

As a corollary of Proposition \ref{Effective second order estimate}, we obtain the following estimate.
\begin{theorem}\label{A priori estimates}
Let $\alpha<0$ and $\vp$ be a smooth solution of (\ref{Fu-Yau equation t}) satisfying (\ref{Elliptic condition t}) and (\ref{L^n normalization condition}). Then there exists a uniform constant $A_{0}$ such that if $A\leq A_{0}$, then we have the following estimate
\begin{equation*}
\|\vp\|_{C^{k}} \leq C_{A,k},
\end{equation*}
where $C_{A,k}$ depends only on $A$, $k$, $\alpha$, $\rho$, $\mu$ and $(M,\omega)$.
\end{theorem}

\begin{proof}
By Proposition \ref{Effective second order estimate} while $\rho$ and $\mu$ are replaced by
\begin{equation*}
t\rho \text{~and~} \frac{n\alpha(t-1)n!\ddbar h\wedge\ddbar h\wedge\omega^{n-2}}{\omega^{n}}+t\mu,
\end{equation*}
we use the blow-up argument to derive (cf. \cite{DK12,Sze15}),
\begin{equation*}
\sup_{M}|\de\db\vp|_{g} \leq C_{A}.
\end{equation*}
By  the $C^{2,\alpha}$-estimate (cf. \cite[Theorem 1.1]{TWWY15}),
it follows
$$\|\varphi\|_{C^{2,\alpha}} \leq C_{A}'.$$
 Hence, by the bootstrapping argument, we complete the proof.
\end{proof}

\section{Proof of Theorem \ref{Uniqueness theorem alpha negative}}
In this section, we give the proof of Theorem \ref{Uniqueness theorem alpha negative}. First we prove the uniqueness of solutions of (\ref{Fu-Yau equation t}) in case of $\alpha<0$ when $t=0$.

\begin{lemma}\label{Uniqueness t=0}
When $t=0$, (\ref{Fu-Yau equation t}) has a unique solution satisfying (\ref{Elliptic condition t}) and (\ref{L^n normalization condition})
\begin{equation*}
\vp_{0} = h+\ln\|e^{-h}\|_{L^{n}}-\ln A,
\end{equation*}
where $h$ is the function in Lemma \ref{Existence lemma}.
\end{lemma}

\begin{proof}
It suffices to prove that $\vp-h$ is constant. For convenience, we define
\begin{equation*}
\ti{\vp} = \vp-h \text{~and~} \omega_{h} = e^{h}\omega.
\end{equation*}
Then, when $t=0$, (\ref{Fu-Yau equation t}) and (\ref{Elliptic condition t}) can be expressed as
\begin{equation*}
\begin{split}
\ddbar(e^{\ti{\vp}}\omega_{h})\wedge\omega^{n-2} & +n\alpha\ddbar\ti{\vp}\wedge\ddbar\ti{\vp}\wedge\omega^{n-2} \\
& +2n\alpha\ddbar\ti{\vp}\wedge\ddbar h\wedge\omega^{n-2} = 0
\end{split}
\end{equation*}
and
\begin{equation*}
e^{\ti{\vp}}\omega_{h}+2n\alpha\ddbar(\ti{\vp}+h) \in \Gamma_{2}(M).
\end{equation*}

By the similar calculation of Lemma \ref{Zero order estimate lemma}, we have
\begin{equation}\label{Uniqueness t=0 equation 1}
\begin{split}
        & \int_{M}e^{2\ti{\vp}}\sqrt{-1}\de\ti{\vp}\wedge\dbar\ti{\vp}\wedge\omega_{h}\wedge\omega^{n-2} \\
\geq {} & -2n\alpha\int_{M}\sqrt{-1}\de e^{\ti{\vp}}\wedge\dbar\ti{\vp}\wedge\ddbar(\ti{\vp}+h)\wedge\omega^{n-2} \\
  =  {} & -2n\alpha\int_{M}e^{\ti{\vp}}\wedge\dbar\ti{\vp}\wedge\ddbar(\ti{\vp}+h)\wedge\sqrt{-1}\de\omega^{n-2} \\
        & +2n\alpha\int_{M}e^{\ti{\vp}}\wedge\ddbar\ti{\vp}\wedge\ddbar(\ti{\vp}+h)\wedge\omega^{n-2} \\
  =  {} & -2n\alpha\int_{M}e^{\ti{\vp}}\ddbar(\ti{\vp}+h)\wedge\ddbar\omega^{n-2} \\
        & -2\int_{M}e^{\ti{\vp}}\ddbar(e^{\ti{\vp}}\omega_{h})\wedge\omega^{n-2} \\
        & -2n\alpha\int_{M}e^{\ti{\vp}}\ddbar\ti{\vp}\wedge\ddbar h\wedge\omega^{n-2}.
\end{split}
\end{equation}
Since $\ddbar\omega^{n-2}=0$, the first term of (\ref{Uniqueness t=0 equation 1}) vanishes. For the second term, we compute
\begin{equation*}
\begin{split}
     & -2\int_{M}e^{\ti{\vp}}\ddbar(e^{\ti{\vp}}\omega_{h})\wedge\omega^{n-2} \\
= {} & 2\int_{M}e^{\ti{\vp}}\sqrt{-1}\de\ti{\vp}\wedge\dbar(e^{\ti{\vp}}\omega_{h})\wedge\omega^{n-2}
       -2\int_{M}e^{\ti{\vp}}\dbar(e^{\ti{\vp}}\omega_{h})\wedge\sqrt{-1}\de\omega^{n-2} \\
= {} & 2\int_{M}(e^{2\ti{\vp}}\sqrt{-1}\de\ti{\vp}\wedge\dbar\ti{\vp}\wedge\omega_{h}
       +\frac{1}{2}\sqrt{-1}\de e^{2\ti{\vp}}\wedge\dbar\omega_{h})\wedge\omega^{n-2} \\
     & -\int_{M}\dbar e^{2\ti{\vp}}\wedge\omega_{h}\wedge\sqrt{-1}\de\omega^{n-2}
       -2\int_{M}e^{2\ti{\vp}}\dbar\omega_{h}\wedge\sqrt{-1}\de\omega^{n-2} \\
= {} &\int_{M}e^{2\ti{\vp}}\left(-\partial(\ol{\partial}\omega_{h}\wedge\omega^{n-1})
       +\ol{\partial}(\omega_{h}\wedge{\partial}\omega^{n-1})-2\ol{\partial}\omega_{h}\wedge\partial\omega^{n-1}\right)\\
     &+2\int_{M}e^{2\ti{\vp}}\sqrt{-1}\de\ti{\vp}\wedge\dbar\ti{\vp}\wedge\omega_{h}\\
= {} & 2\int_{M}e^{2\ti{\vp}}\sqrt{-1}\de\ti{\vp}\wedge\dbar\ti{\vp}\wedge\omega_{h}\wedge\omega^{n-2},
\end{split}
\end{equation*}
where we used the relations in the last equality,
$$\ddbar\omega^{n-2}=0~{\rm and}~ \ddbar(e^{h}\omega)\wedge\omega^{n-2}=0.$$  For the third term of (\ref{Uniqueness t=0 equation 1}), we see that
\begin{equation*}
\begin{split}
     & -2n\alpha\int_{M}e^{\ti{\vp}}\ddbar\ti{\vp}\wedge\ddbar h\wedge\omega^{n-2} \\
= {} & 2n\alpha\int_{M}e^{\ti{\vp}}\sqrt{-1}\de\ti{\vp}\wedge\dbar\ti{\vp}\wedge\ddbar h\wedge\omega^{n-2} \\
     & -2n\alpha\int_{M}\dbar e^{\ti{\vp}}\wedge\ddbar h\wedge\sqrt{-1}\de\omega^{n-2} \\
= {} & 2n\alpha\int_{M}e^{\ti{\vp}}\sqrt{-1}\de\ti{\vp}\wedge\dbar\ti{\vp}\wedge\ddbar h\wedge\omega^{n-2}.
\end{split}
\end{equation*}
 Substituting the above estimates into (\ref{Uniqueness t=0 equation 1}), we get the inequality
\begin{equation}\label{Uniqueness t=0 equation 2}
\begin{split}
        & \int_{M}e^{2\ti{\vp}}\sqrt{-1}\de\ti{\vp}\wedge\dbar\ti{\vp}\wedge\omega_{h}\wedge\omega^{n-2} \\
\leq {} & -2n\alpha\int_{M}e^{\ti{\vp}}\sqrt{-1}\de\ti{\vp}\wedge\dbar\ti{\vp}\wedge\ddbar h\wedge\omega^{n-2} \\
\leq {} & C\int_{M}e^{\ti{\vp}}\sqrt{-1}\de\ti{\vp}\wedge\dbar\ti{\vp}\wedge\omega_{h}\wedge\omega^{n-2}.
\end{split}
\end{equation}
On the other hand, by Proposition \ref{Stronger zero order estimate},  we have
\begin{equation}\label{tilde-phi}
e^{\ti{\vp}} = e^{\vp-h} \geq \frac{1}{CA}.
\end{equation}
Combining this with (\ref{Uniqueness t=0 equation 2}) and $A\ll1$, we prove
\begin{equation*}
\int_{M}e^{2\ti{\vp}}\sqrt{-1}\de\ti{\vp}\wedge\dbar\ti{\vp}\wedge\omega_{h}\wedge\omega^{n-2} = 0,
\end{equation*}
which implies $\de\ti{\vp}=0$. Therefore, $\ti{\vp}$ is constant.
\end{proof}

\begin{remark}\label{unique-2-h} Lemma \ref{Uniqueness t=0} is also   true from   the above  proof  if (\ref{L^n normalization condition}) is replaced by the condition (\ref{two-condition}),  and the solution of  (\ref{Fu-Yau equation t})   when $t=0$ is given by $\vp_{0} = h+\ln\|e^{-h}\|_{L^{1}}-\ln A$.    The reason is  that (\ref{tilde-phi})  holds by Proposition \ref{Zero order estimate}. In this case, the lemma holds for any $\alpha\neq0$.

\end{remark}

Now we are in a position to prove Theorem \ref{Uniqueness theorem alpha negative}.

\begin{proof}[Proof of Theorem \ref{Uniqueness theorem alpha negative}]
The existence is proved in Theorem \ref{Existence and Uniqueness Theorem}. It suffices to prove the uniqueness. Assume that we have two solutions $\vp$ and $\vp'$.  Then as in  \cite[Section 5.3]{CHZ18},    we use the continuity method to solve (\ref{Fu-Yau equation t}) from $t=1$ to $0$.   By Theorem \ref{A priori estimates},    there are two families of solutions $\{\vp_{t}\}$ and $\{\vp_{t}'\}$ of (\ref{Fu-Yau equation t}) satisfying (\ref{Elliptic condition t}) and $L^{n}$-normalization conditions. We also have $\vp_{1}=\vp$ and $\vp_{1}'=\vp'$. By Lemma \ref{Uniqueness t=0}, we see that
\begin{equation*}
\vp_{0} = \vp_{0}' = -\ln A.
\end{equation*}
Let
\begin{equation*}
J = \{ t\in [0,1] ~|~ \vp_{t}=\vp_{t}'\}.
\end{equation*}
Clearly, $J$ is closed. Applying the implicit function theorem, we see that $J$ is open. Then $J=[0,1]$ and so
\begin{equation*}
\vp = \vp_{1} = \vp_{1}' = \vp'.
\end{equation*}
This completes the proof of uniqueness.
\end{proof}

\section{Proof of Theorem \ref{Monotonicity theorem}}
In this section, we give the proof of Theorem \ref{Monotonicity theorem}. When $A$ is sufficiently small, Theorem \ref{Uniqueness theorem alpha negative} implies that there exists a unique solution of (\ref{Fu-Yau equation t}) satisfying (\ref{Elliptic condition t}) and (\ref{L^n normalization condition}). For convenience, we denote it by $\vp_{t,A}$.

\begin{lemma}\label{Smooth lemma}
$\vp_{t,A}$ is smooth with respect to $t$ and $A$.
\end{lemma}

\begin{proof}
For $\beta\in(0,1)$, we define the sets
\begin{equation*}
\begin{split}
\ti{B} & = C^{2,\beta}(M)\times[0,1]\times[0,1],\\[2mm]
\ti{B}_{1} & = \{ (\vp,t,A)\in \ti{B} ~|~ \text{$\vp$ satisfies (\ref{Elliptic condition t})} \}, \\
\ti{B}_{2} & = \textbf{R}\times\{ u\in C^{\beta}(M) ~|~ \int_{M}u\omega^{n}=0 \},
\end{split}
\end{equation*}
and map $\Phi: \ti{B}_{1}\rightarrow \ti{B}_{2}$
\begin{equation*}
\Phi(\vp,t,A) = \left(\Phi_{1}(\vp,A),\Phi_{2}(\vp,t)\right),
\end{equation*}
where
\begin{equation*}
\Phi_{1}(\vp,A) = \int_{M}e^{-n\vp}\omega^{n}-A^{n}
\end{equation*}
and
\begin{equation*}
\begin{split}
\Phi_{2}(\vp,t) = {} & \ddbar(e^{\vp}\omega-t\alpha e^{-\vp}\rho)\wedge\omega^{n-2} \\[2mm]
                     & +n\alpha\ddbar\vp\wedge\ddbar\vp\wedge\omega^{n-2} \\
                     & +n\alpha(t-1)\ddbar h\wedge\ddbar h\wedge\omega^{n-2}+t\mu\frac{\omega^{n}}{n!}.
\end{split}
\end{equation*}

By the same argument of Section \ref{Proof of existence theorem}, we have
\begin{equation*}
(D_{\vp}\Phi)_{(\hat{\vp},\hat{t},\hat{A})}(u) = \left(-\int_{M}ne^{-n\hat{\vp}}u\omega^{n},Lu\right),
\end{equation*}
where
\begin{equation*}\label{Definition of L}
\begin{split}
(Lu)\omega^{n} = {} & \ddbar(ue^{\hat{\vp}}\omega+\hat{t}\alpha ue^{-\hat{\vp}}\rho)\wedge\omega^{n-2} \\
                    & +2n\alpha\ddbar\hat{\vp}\wedge\ddbar u\wedge\omega^{n-2}.
\end{split}
\end{equation*}
and
\begin{equation}\label{Kernel of L}
\textrm{Ker}L = \{ cu_{0} ~|~ c\in\mathbf{R} \},
\end{equation}
where $u_{0}$ is a positive function. Similarly, we see that $(D_{\vp}\Phi)_{(\hat{\vp},\hat{t},\hat{A})}$ is invertible. Using the implicit function theorem, near $(\hat{\vp},\hat{t},\hat{A})$, there exists a smooth map $F(t,A)$ such that $\Phi(F(t,A),t,A)=(0,0)$, where $\vp_{\hat{t},\hat{A}}$ (for convenience, we denote it by $\hat{\vp}$) satisfies
\begin{equation*}
\Phi(\hat{\vp},\hat{t},\hat{A}) = (0,0).
\end{equation*}
This implies that $F(t,A)$ is the solution of (\ref{Fu-Yau equation t}) satisfying the elliptic and $L^{n}$-normalization conditions. Thanks to Theorem \ref{Uniqueness theorem alpha negative}, we have $\vp_{t,A}=F(t,A)$. Hence, $\vp_{t,A}$ is smooth at $(\hat{t},\hat{A})$. Since $(\hat{t},\hat{A})$ is arbitrary, we complete the proof.
\end{proof}

\begin{proof}[Proof of Theorem \ref{Monotonicity theorem}]
By the definition of $\vp_{t,A}$, it suffices that prove $\vp_{1,A}>\vp_{1,\ti{A}}$.  Define $\vp(s)=\vp_{1,A^{s}\ti{A}^{1-s}}$.  Then it  follows that
\begin{equation*}
\begin{split}
\ddbar(e^{\vp(s)}\omega & -t\alpha e^{-\vp(s)}\rho)\wedge\omega^{n-2}+n\alpha\ddbar\vp(s)\wedge\ddbar\vp(s)\wedge\omega^{n-2} \\
& +n\alpha(t-1)\ddbar h\wedge\ddbar h\wedge\omega^{n-2}+t\mu\frac{\omega^{n}}{n!} = 0,
\end{split}
\end{equation*}
and
\begin{equation*}
\int_{M}e^{-n\vp(s)}\omega^{n} = A^{ns}\ti{A}^{n(1-s)}.
\end{equation*}
By Lemma \ref{Smooth lemma},   we can differentiate the above two  equations with respect to $s$,  respectively,  and we obtain
\begin{equation*}
\frac{\de\vp(s)}{\de s} \in \textrm{Ker}L \text{~and~} \int_{M}e^{-n\vp(s)}\frac{\de\vp(s)}{\de s}\omega^{n} > 0.
\end{equation*}
Recalling (\ref{Kernel of L}), we see that $\frac{\de\vp(s)}{\de s} > 0$, which implies
\begin{equation*}
\vp_{1,A}-\vp_{1,\ti{A}} = \int_{0}^{1}\frac{\de\vp(s)}{\de s}ds > 0.
\end{equation*}
Theorem \ref{Monotonicity theorem} is proved.
\end{proof}

\begin{remark}\label{Special case remark}
When $\textrm{tr}_{\omega}\rho\geq0$, using (\ref{Zero order estimate equation 2}) and Sobolev inequality, we obtain
\begin{equation*}
\left(\int_{M}e^{-k\beta\vp}\omega^{n}\right)^{\frac{1}{\beta}}
\leq Ck\int_{M}e^{-(k-1)\vp}\omega^{n},
\end{equation*}
which implies $\|e^{-\vp}\|_{L^{\infty}}\leq C$. By the similar argument in Section 5, we get the analogous estimate of Theorem \ref{A priori estimates} under the normalization $\|e^{-\vp}\|_{L^{1}}=A$. This estimate is enough for the proofs of Theorem \ref{Uniqueness theorem alpha negative} and \ref{Monotonicity theorem}.
\end{remark}

\end{document}